\documentclass[12pt]{elsarticle}
\usepackage{lineno,hyperref} 
\modulolinenumbers[5]
\usepackage{amsmath, amssymb, amsthm}
\usepackage{tikz}
\usepackage{graphicx,subfigure}
\usepackage{nameref}
\usetikzlibrary{arrows,shapes,chains}
\usepackage[a4paper, left=3cm, right=3cm, top=2.5cm, bottom=2.5cm]{geometry}

\newtheorem{thm}{Theorem}[section]

\newtheorem{lem}[thm]{Lemma}
\newtheorem{pro}[thm]{Proposition}
\newtheorem{rem}{Remark}

\newtheorem*{thmA}{Theorem A}
\newtheorem*{thmB}{Theorem B}
\newtheorem*{thmC}{Theorem C}
\newtheorem*{thmD}{Theorem D}
\newtheorem*{thmE}{Theorem E}
\newtheorem*{thmF}{Theorem F}
\newtheorem*{thmG}{Theorem G}
\newtheorem*{que}{Question}
\newtheorem*{conj}{Conjecture}

\numberwithin{equation}{section}
\journal{Journal of \LaTeX\ Templates}

\begin{document}

\begin{frontmatter}

\title{Restricted sumsets in $\mathbb{Z}$\tnoteref{mytitlenote}}
\tnotetext[mytitlenote]{This work was supported by the National Natural Science Foundation of China(Grant
Nos. 12101007 and 12371003) and the Natural Science Foundation of Anhui Province (Grant No. 2008085QA06).}

\author[mymainaddress]{Yujie Wang\corref{mycorrespondingauthor}}
\ead{wangyujie9291@126.com}

\author[mymainaddress]{Min Tang}
\ead{tmzzz2000@163.com}
\cortext[mycorrespondingauthor]{Corresponding author}

\address[mymainaddress]{School of Mathematics and Statistics, Anhui Normal University, Wuhu 241002, P. R. China}

\begin{abstract}
  Let $k\geqslant 3$ and let $A=\{0=a_{0}<a_{1}<\cdots<a_{k-1}\}$ with $\gcd(A)=1$.
   Freiman-Lev conjecture
  [V.F. Lev, Restricted set addition in groups, I. The classical setting, J. London Math. Soc. 62(2000), 27-40] is a well-known conjecture which related to restricted sumsets.
  Up to now, Freiman-Lev conjecture is open for all $a_{k-1}\geqslant 2k-2$.
  In this paper, we prove the Freiman-Lev conjecture is true for $a_{k-1}\geqslant 2k-2$ and $a_{k-2}<2k-4$.
  That is, Freiman-Lev conjecture is still open for the case $a_{k-1}\geqslant 2k-2$ and $a_{k-2}\geq 2k-4$.
\end{abstract}


\end{frontmatter}


\section{Introduction}
If $A$ and $B$ are two subsets of some arbitrary group $G$, the sumset and restricted sumset of $A$, $B$ are defined as
$$A+B=\{a+b: a\in A, b\in B\}, \ \ A\widehat{+}B=\{a+b: a\in A, b\in B, a\neq b\}. $$
In other words, the set $A\widehat{+}B$ is obtained from $A+B$ by excluding those sums with $b=a$.
There are many famous results regarding sumsets, such as the Cauchy-Davenport theorem and Vosper's theorem.

When $G=\mathbb{Z}$, let
$2A$ and $2^{\wedge}A$ denote the set of all sums of two elements of $A$ and the set of all sums of two distinct elements of $A$, respectively.
Define the interval of integers $[a, b]=\{x\in \mathbb{Z}\ |\ a\leqslant x\leqslant b\}$ and $\gcd(A)$ the greatest common divisor of all nonzero elements of $A$.

The following fundamental and important result of sumset $2A$ is immediately available.

\begin{thmA}
Let $A$ be a set of $k$ integers. Then
$$|2A|\geqslant 2k-1. $$
Moreover, $|2A|=2k-1$ if and only if $A$ is a $k$-term arithmetic progression.
\end{thmA}

G.A. Freiman \cite{Freiman1, Freiman} obtained the famous Freiman's $2A$ theorem.
\begin{thmB}
Let $A$ be a set of $k\geqslant 3$ integers such that $A\subseteq [0, l]$, $0, l\in A$ and $\gcd(A)=1$. Then
$$|2A|\geqslant \left\{\begin{array}{ll}
l+k &\text{ if } l\leqslant 2k-3, \\
3k-3 &\text{ if } l\geqslant 2k-2. \end{array}\right.$$
\end{thmB}

G.A. Freiman \cite{Freiman} described the structure of $A$ if $|2A|$ is not much greater than the minimal value, which is well-known as Freiman's $3k-4$ theorem.

\begin{thmC}
Let $A$ be a set of $k\geqslant 3$ integers. If
$$|2A|=2k-1+b\leqslant 3k-4, $$
then $A$ is a subset of an arithmetic progression of length at most $k+b\leqslant 2k-3$.
\end{thmC}

G.A. Freiman \cite{Freiman1} also described the structure of $A$ for $|2A|=3k-3$.
\begin{thmD}
Let $A$ be a set of $k>6$ integers such that $A\subseteq [0, l]$, $0, l\in A$ and $\gcd(A)=1$.
If $|2A|=3k-3$, then $A$ is a subset of an arithmetic progression of length at most $2k-1$ or a union of two arithmetic progressions with the same common difference.
\end{thmD}

Continue these works, Freiman \cite{Freiman64} described the structure of $A$ when $|2A|<c|A|$, where $c$ is any given positive number, which known as the famous Freiman's theorem, an improved version of a proof was presented in \cite{Freiman3}.
Freiman \cite{Freiman2} himself remarked the comparative difficulty of the above results as ``
To prove Theorem A took one minute and Theorem C was studied in three minutes.
The proof of Theorem D took one month.
Proof of the Freiman's Theorem took five years."

 Many well-known results concerning sumsets should have analogues for restricted sumsets, but in many cases such results are difficult to establish or are not known at all.
For instance, the Cauchy-Davenport theorem, which was established by A.L. Cauchy as early as 1813 and rediscovered by H. Davenport in 1935,
the analogue restricted result which known as the Erd\H{o}s-Heilbronn conjecture, was not proven until 1994 by J. Dias da Silva and Y.O. Hamidoune.

For the restricted sumset $2^{\wedge}A$, V.F. Lev \cite{Lev2} remarked that the following conjecture (in personal communication with Freiman).
\begin{conj} Let $A$ be a set of $k>7$ integers such that $A\subseteq [0, l]$, $0, l\in A$ and $\gcd(A)=1$. Then
$$|2^{\wedge}A|\geqslant \left\{\begin{array}{ll}
l+k-2 &\text{ if } l\leqslant 2k-5, \\
3k-7 &\text{ if } l\geqslant 2k-4. \end{array}\right.$$
\end{conj}

G.A. Freiman, L. Low and J. Pitman \cite{Freiman1999} obtained the following result.
\begin{thmE}
Let $A$ be a set of $k\geqslant 3$ integers such that $A\subseteq [0, l]$, $0, l\in A$ and $\gcd(A)=1$. Then
$$|2^{\wedge}A|\geqslant \left\{\begin{array}{ll}
0.5(l+k)+k-3.5 &\text{ if } l\leqslant 2k-3, \\
2.5k-5 &\text{ if } l\geqslant 2k-2. \end{array}\right.$$
\end{thmE}

V.F. Lev \cite{Lev2} obtained a result nearer to the above conjecture by himself.
\begin{thmF}
Let $A$ be a set of $k\geqslant 3$ integers such that $A\subseteq [0, l]$, $0, l\in A$ and $\gcd(A)=1$. Then
$$|2^{\wedge}A|\geqslant \left\{\begin{array}{ll}
l+k-2 &\text{ if } l\leqslant 2k-5, \\
(\theta+1)k-6 &\text{ if } l\geqslant 2k-4, \end{array}\right.$$
where $\theta=(1+\sqrt{5})/2$.
\end{thmF}

T. Schoen \cite{Sch} gave some progress of the above conjecture by combining a result of I. Z. Ruzsa \cite{Ruzsa} with a theorem of J. Bourgain  \cite{Bour}.
For various research, the reader can see \cite{Nathanson1} and \cite{Tang}.
For over twenty years, Freiman-Lev conjecture has remained an open unsolved problem when $l\geqslant 2k-4$.
Recently, the authors of this paper \cite{Wang} have proven that Freiman-Lev conjecture holds for $l = 2k-4$ and $l = 2k-3$.

\begin{thmG}
Let $A$ be a set of $k\geqslant 5$ integers such that $A\subseteq [0, l]$, $0, l\in A$ and $\gcd(A)=1$.
If $2k-4\leqslant l\leqslant 2k-3$, then $|2^{\wedge}A|\geqslant 3k-7$.
\end{thmG}

In this paper, we give the new progress on the Feriman-Lev conjecture.

\begin{thm}\label{T1}
Let $k\geqslant 3$.
Let $A=\{0=a_{0}<a_{1}<\cdots<a_{k-1}\}$ such that $a_{k-2}<2k-4$, $a_{k-1}\geqslant 2k-2$ and $\gcd(A)=1$.
Then $|2^{\wedge}A|\geqslant 3k-7$.
\end{thm}

The main method of Freiman's $2A$ theorem is to make the following classification discussion:
\begin{itemize}
  \item $a_{i}<2i$ for each integer $1\leqslant i\leqslant k-2$;
  \item $a_{k-2}<2k-4$, but $a_{i}\geqslant 2i$ for some integer $1\leqslant i\leqslant k-3$;
  \item   $a_{k-2}\geqslant 2k-4$.
\end{itemize}

In this paper, we deal the first two cases for restricted sumsets.
The first case is described as Theorem \ref{T2} and we determine the structure of $A$ when $|2^{\wedge}A|=3k-7$ under those restrictive conditions.
To solve the second case, we characterize the structure of $A$ satisfying $|2^{\wedge}A|=3k-7$ in the case $a_{k-1}=2k-3$, which is described as Theorem \ref{T3}.

\begin{thm}\label{T2}
Let $k\geqslant 3$.
Let $A=\{0=a_{0}<a_{1}<\cdots<a_{k-1}\}$ such that $a_{i}<2i$ for all $i=1, \ldots, k-2$, $a_{k-1}\geqslant 2k-2$ and $\gcd(A)=1$.
Then
$$|2^{\wedge}A|\geqslant 3k-7. $$
Moreover, $|2^{\wedge}A|=3k-7$ if and only if $k\geqslant 6$ and one of the following cases holds:

(1) $k\equiv 0\pmod{3}$ and
$$A=\left([0,k-3]\cap 3\mathbb{Z}\right)\cup \left([1,2k-2]\cap (3\mathbb{Z}+1)\right);$$

(2) $k\equiv 1\pmod{3}$ and
$$A=\left([0,2k-2]\cap 3\mathbb{Z}\right)\cup \left([1,k-3]\cap (3\mathbb{Z}+1)\right).$$
\end{thm}

\begin{rem}\label{r2.1}
Let $A$ be as in Theorem \ref{T2}. If $|2^{\wedge}A|=3k-7$, then
$$2^{\wedge}A=\left([1, 2k-4]\backslash \{2, 2k-6\}\right)\cup \left(a_{k-1}+A\backslash\{a_{k-1}\}\right). $$
\end{rem}

\begin{thm}\label{T3}
Let $k\geqslant 3$. Let $A=\{0=a_{0}<a_{1}<\cdots<a_{k-1}=2k-3\}$.
Then $|2^{\wedge}A|=3k-7$ if and only if $k\geqslant 4$ and one of the following cases holds:
\begin{itemize}
  \item $A=[0, \theta-k+1]\cup [\theta, 2k-3]$, $k\leqslant \theta\leqslant 2k-4$;
  \item $A=\{2i, 2j-1: i\in [0, \theta], j\in[\theta+1, k-1]\}$, $1\leqslant \theta\leqslant k-3$;
  \item $A=\{3i, 3j-k: i\in [0, \theta], j\in[\theta+1, k-1]\}$, $\frac{k-3}{3}<\theta<\frac{2k-3}{3}$, $3\nmid k$;
  \item $A=\left\{0, 4i, 4i-3, 2k-3: 1\leqslant i\leqslant \frac{k-2}{2}\right\}$, $2\mid k$;
  \item $A=\left\{3i: 0\leqslant i\leqslant \frac{2k-3}{3}\right\}\cup\left\{\theta+3i: 0\leqslant i\leqslant \frac{k-3}{3}\right\}$, $1\leqslant\theta\leqslant k-1$, $3\nmid\theta$, $3\mid k$;
  \item $A=\{0, 1, 4, 5, 6, 9\}$, $\{0, 3, 4, 5, 8, 9\}$, $\{0, 1, 2, 5, 6, 7, 11\}$,

\qquad$\{0, 1, 3, 4, 7, 8, 11\}$, $\{0, 1, 4, 5, 6, 10, 11\}$, $\{0, 1, 4, 5, 7, 8, 11\}$,

\qquad$\{0, 1, 5, 6, 7, 10, 11\}$, $\{0, 3, 4, 6, 7, 10, 11\}$, $\{0, 3, 4, 7, 8, 10, 11\}$,

\qquad$\{0, 4, 5, 6, 9, 10, 11\}$, $\{0, 1, 2, 6, 7, 8, 12, 13\}$, $\{0, 1, 4, 5, 6, 9, 10, 13\}$,

\qquad$\{0, 1, 5, 6, 7, 8, 12, 13\}$, $\{0, 1, 5, 6, 7, 11, 12, 13\}$, $\{0, 2, 3, 5, 7, 8, 10, 13\}$,

\qquad$\{0, 2, 3, 5, 8, 10, 11, 13\}$, $\{0, 3, 4, 7, 8, 9, 12, 13\}$, $\{0, 3, 5, 6, 8, 10, 11, 13\}$,

\qquad$\{5i, 5i+\theta, 15: i=0, 1, 2\}\cup \{5-\theta, 10-\theta\}$, $1\leqslant \theta\leqslant 4$,

\qquad$\{5i, 5i+\theta, 15: i=0, 1, 2\}\cup \{10-\theta, 15-\theta\}$, $1\leqslant \theta\leqslant 4$,

\qquad$\{0, 1, 2, 6, 7, 8, 9, 14, 15\}$, $\{0, 1, 4, 5, 7, 8, 11, 12, 15\}$,

\qquad$\{0, 1, 6, 7, 8, 9, 13, 14, 15\}$, $\{0, 3, 4, 7, 8, 10, 11, 14, 15\}$,

\qquad$\{0, 1, 2, 7, 8, 9, 10, 15, 16, 17\}$, $\{0, 1, 5, 6, 7, 10, 11, 12, 16, 17\}$,

\qquad$\{0, 2, 3, 5, 7, 8, 10, 12, 15, 17\}$, $\{0, 2, 5, 7, 9, 10, 12, 14, 15, 17\}$,

\qquad$\{0, 3, 4, 6, 10, 11, 13, 14, 17\}$.
\end{itemize}
\end{thm}

\begin{rem}\label{r1}
Let $A$ be as in Theorem \ref{T3} such that $|2^{\wedge}A|=3k-7$.
If $a_{k-3}<2k-6$, $a_{k-2}=2k-4$ and $a_{k-1}=2k-3$, then one of the following cases holds:

\begin{itemize}
  \item $A=[0, k-3]\cup \{2k-4, 2k-3\}$, $k\geqslant 4$;

  \item $A=\{0, 3, 4, 5, 8, 9\}$, $\{0, 1, 4, 5, 6, 10, 11\}$, $\{0, 1, 5, 6, 7, 10, 11\}$,

\qquad$\{0, 3, 4, 6, 7, 10, 11\}$, $\{0, 1, 2, 6, 7, 8, 12, 13\}$, $\{0, 1, 5, 6, 7, 8, 12, 13\}$,

\qquad$\{0, 3, 4, 7, 8, 9, 12, 13\}$, $\{0, 1, 2, 6, 7, 8, 9, 14, 15\}$,

\qquad$\{0, 1, 4, 5, 6, 9, 10, 14, 15\}$, $\{0, 1, 5, 6, 9, 10, 11, 14, 15\}$,

\qquad$\{0, 4, 5, 6, 9, 10, 11, 14, 15\}$, $\{0, 3, 4, 7, 8, 10, 11, 14, 15\}$,

\qquad$\{0, 1, 5, 6, 7, 10, 11, 12, 16, 17\}$.
\end{itemize}
\end{rem}

\begin{rem}\label{r2}
Let $A$ be as in Theorem \ref{T3} such that $|2^{\wedge}A|=3k-7$.
If $a_{k-3}=2k-6$, $a_{k-2}=2k-5$ and $a_{k-1}=2k-3$, then $a_{k-4}=2k-8$.
\end{rem}

From the results mentioned above, it can be observed that under some certain specific conditions,
$|2^{\wedge}A|=3k-7$ seems to suggest that the structure of $A$ is a union of arithmetic progressions with the same common difference.
This naturally leads us to reflect whether there is something analogous to Freiman's work for restricted sumsets?

\begin{que}
Let $A$ be as in Theorem \ref{T1}. When $|2^{\wedge}A|=3k-7$, is $A$ a union of arithmetic progressions with the same common difference?
\end{que}

Our key tool is to split the set $A$ into two parts, each meeting the conditions required by Theorems \ref{T2} and \ref{T3}.
For the sake of completeness we present the proofs of them in Sections \ref{S5} and \ref{S6}.
In Section \ref{S7}, let us see how to successively apply Theorems \ref{T2} and \ref{T3} to prove Theorem \ref{T1}.

\section{Preliminaries of Theorem \ref{T2}}\label{S2}
Throughout this section, let $A=\left\{0=a_{0}<a_{1}<\cdots<a_{k-1}\right\}$ be a set of integers such that
$$a_{i}<2i, \ \ i=1, \ldots, k-2, \ \ a_{k-1}\geqslant 2k-2, \ \ \gcd(A)=1. $$
Write $A'=A\backslash\{a_{k-1}\}$, $b_{0}=0$ and
$$B=[1, 2k-4]\backslash 2^{\wedge}A':=\{b_{1}<b_{2}<\cdots<b_{m}\}. $$

If $m\leqslant 1$, then
$$|2^{\wedge}A|\geqslant |[1, 2k-4]\backslash B|+|a_{k-1}+A'|\geqslant 3k-6. $$
Now, we consider that $m\geqslant 2$.

To count $|2^{\wedge}A|$, we shall consider the following mutually disjoint parts:
$$[1, 2k-4]\cap 2^{\wedge}A', \ \ 2k-4+[1, b_{m-1}], $$
$$\{2k-3+b_{m-1}, 2k-2+b_{m-1}\}, \ \ a_{k-1}+([b_{m-1}+1, a_{k-2}]\cap A'). $$
Except for finitely exceptions, we shall show that
$$|2k-4+[1,b_{m-1}]\cap 2^{\wedge}A'|\geq \frac{b_{m-1}}{2}+\left\lfloor\frac{b_{m-1}}{4}\right\rfloor,$$
$$|\{2k-3+b_{m-1}, 2k-2+b_{m-1}\}\cap 2^{\wedge}A'|\geq 1.$$

To increase the readability of the paper, we put the proof of Lemmas \ref{P1}-\ref{P3} and Proposition \ref{P4} in the appendix.

\begin{lem}\label{L2-1}(\cite{Nathanson}, Theorem 1.14)
$[0, 2k-4]\subseteq 2A'$.
\end{lem}

\begin{lem}\label{L2-2}
For any integer $b\in B$, we have

(i) $b\not\in A'$, $2\mid b$, $\frac{b}{2}\in A'$;

(ii) $\left|[0, b]\cap A'\right|=\frac{b}{2}+1$. Moreover, $\left|\left\{i, b-i\right\}\cap A'\right|=1$ for all $i=0, 1, \ldots, \frac{b}{2}$;

(iii) Let $u$, $b$ be integers such that $b<k-2$ and $1\leqslant u\leqslant b$. If $2k-4+u\not\in 2^{\wedge}A'$, then
$$|[b+1, 2k-5+u-b]\cap A'|=k-2+\left\lfloor\frac{u}{2}\right\rfloor-b, $$
$$|\{i, 2k-4+u-i\}\cap A'|=1, \ \ i=b+1, \ldots, k-2+\left\lfloor \frac{u}{2}\right\rfloor. $$

(iv) If $b<2k-4$, then $a_{\frac{b}{2}+1}=b+1\in A'$.
\end{lem}

\begin{proof}
(i) If $b\in A'$, then by $0\in A'$, we have $b\in 2^{\wedge}A'$, which is impossible. Thus $b\not\in A'$.
By Lemma \ref{L2-1}, we have $B\subseteq 2A'\backslash 2^{\wedge}A'$, thus $2\mid b$ and $\frac{b}{2}\in A'$.

(ii) Since $b\not\in 2^{\wedge}A'$, we have $\left|\left\{i, b-i\right\}\cap A'\right|\leqslant 1$ for all $i=0, 1, \ldots, \frac{b}{2}$,
thus $$\left|[0, b]\cap A'\right|\leqslant\frac{b}{2}+1. $$
Since $b\leqslant 2k-4$ and $a_{i}<2i$ for all $i=1, \ldots, k-2$, we have $a_{\frac{b}{2}}<b$, thus
$$\left|[0, b]\cap A'\right|\geqslant\frac{b}{2}+1. $$
Hence
$$\left|[0, b]\cap A'\right|=\frac{b}{2}+1. $$
It follows that
$$\left|\left\{i, b-i\right\}\cap A'\right|=1, \ \ i=0, 1, \ldots, \frac{b}{2}. $$

(iii) The proof is similar to (2). Here is omitted.

(iv) If $b<2k-4$, then $\frac{b}{2}+1\leqslant k-2$.
Since
$$a_{\frac{b}{2}+1}<b+2, \ \ a_{\frac{b}{2}+1}\not\in [0, b], $$
we have $a_{\frac{b}{2}+1}=b+1$.

This completes the proof of Lemma \ref{L2-2}.
\end{proof}

\begin{lem}\label{L2-3}
For any $i\in \{0, \ldots, m-1\}$, we have $b_{i+1}\geqslant 2b_{i}+2$.
\end{lem}

\begin{proof}
Since $a_{1}<2$, we have $a_{1}=1$, thus
$$b_{1}\geqslant 2=2b_{0}+2. $$

For any integer $i\in \{1, \ldots, m-1\}$, by Lemma \ref{L2-2} (i),
we have $b_{i}\not\in A'$, $\frac{b_{i+1}}{2}\in A'$, thus $b_{i+1}\neq 2b_{i}$.
Suppose, for a contradiction, that there exists an integer $i\in \{1, \ldots, m-1\}$ such that $b_{i+1}\leqslant 2b_{i}-2$.

Noting that $\frac{b_{i+1}}{2}\in [b_{i+1}-b_{i}, b_{i}]$, assume that $q$ is a positive integer such that
$$q(b_{i+1}-b_{i})\leqslant \frac{b_{i+1}}{2}<(q+1)(b_{i+1}-b_{i}). $$
Choosing $b:=b_i\rightarrow b:=b_{i+1}\rightarrow b:=b_i\rightarrow \cdots\rightarrow b:=b_i$, when $b:=b_i$ occurs $q$ times, the above process terminates.
By Lemma \ref{L2-2} (ii) we obtain
\begin{eqnarray}\label{e2.1}\nonumber
\frac{b_{i+1}}{2}\in A' &\Longleftrightarrow& b_{i}-\frac{b_{i+1}}{2}=\frac{b_{i+1}}{2}-(b_{i+1}-b_{i})\not\in A'\\ \nonumber
&\Longleftrightarrow& b_{i+1}-\left(b_{i}-\frac{b_{i+1}}{2}\right)\in A'\\
&\Longleftrightarrow& \frac{b_{i+1}}{2}-2(b_{i+1}-b_{i})\not\in A' \\ \nonumber
&\cdots&\\ \nonumber
&\Longleftrightarrow& \frac{b_{i+1}}{2}-q(b_{i+1}-b_{i})\not\in A'. \nonumber\end{eqnarray}

For any integer $n\in [b_{i+1}-b_{i}, b_{i}]$ with $n\not\equiv \frac{b_{i+1}}{2}\pmod{b_{i+1}-b_{i}}$, there exists a positive integer $q_{n}$ such that
$$q_{n}(b_{i+1}-b_{i})\leqslant n<(q_{n}+1)(b_{i+1}-b_{i}). $$
Repeat the procedure ``$b:=b_{i+1}\rightarrow b:=b_{i}$" $q_n$ times, by Lemma \ref{L2-2} (ii) we obtain
\begin{eqnarray}\label{e2.2}\nonumber n\in A' &\Longleftrightarrow& b_{i+1}-n\not\in A'\\ &\Longleftrightarrow &n-(b_{i+1}-b_{i})=b_{i}-(b_{i+1}-n)\in A' \\
  &\cdots& \nonumber \\ &\Longleftrightarrow & n-q_{n}(b_{i+1}-b_{i})\in A'.\nonumber\end{eqnarray}

By the division algorithm, there exist integers $s$ and $\overline{b_{i}}$ such that
\begin{equation}\label{e2.3}b_{i}=s(b_{i+1}-b_{i})+\overline{b_{i}}, \ \ 0\leqslant \overline{b_{i}}<b_{i+1}-b_{i}. \end{equation}

If $s$ is odd, then
$$\frac{\overline{b_{i}}}{2}=\frac{b_{i}}{2}-\frac{s+1}{2}(b_{i+1}-b_{i})+\frac{b_{i+1}-b_{i}}{2}=\frac{b_{i+1}}{2}-\frac{s+1}{2}(b_{i+1}-b_{i}). $$
Since $\frac{b_{i+1}}{2}\in A'$, by (\ref{e2.1}) we have $\frac{\overline{b_{i}}}{2}\not\in A'$.

For $j=0, \ldots, \frac{\overline{b_{i}}}{2}-1$, we have
\begin{equation}\label{e2.4}\overline{b_{i}}-j\not\equiv \frac{b_{i+1}}{2}\pmod{b_{i+1}-b_{i}}. \end{equation}
By (\ref{e2.3}) we have
\begin{equation}\label{e2.5}b_{i}-j=\overline{b_{i}}-j+s(b_{i+1}-b_{i}).\end{equation}
By (\ref{e2.2}), (\ref{e2.4})-(\ref{e2.5}), we know that $\overline{b_{i}}-j, b_{i}-j\in A',$ or $\overline{b_{i}}-j, b_{i}-j\not\in A'.$
Since $j<\frac{\overline{b_{i}}}{2}<\frac{b_i}{2}$, by Lemma \ref{L2-2} (ii) we have $|\{j, b_i-j\}\cap A'|=1$. Thus
$$|\{j, \overline{b_{i}}-j\}\cap A'|=1, \ \ j=0, \ldots, \frac{\overline{b_{i}}}{2}-1. $$
That is, $$\left|\left[0, \overline{b_{i}}\right]\cap A'\right|=\frac{\overline{b_{i}}}{2}.$$
Hence, $a_{\frac{\overline{b_{i}}}{2}}>\overline{b_{i}}$, a contradiction.

If $s$ is even, then
$$\frac{\overline{b_{i}}+(b_{i+1}-b_{i})}{2}=\frac{b_{i+1}}{2}-\frac{s}{2}(b_{i+1}-b_{i}). $$
Since $\frac{b_{i+1}}{2}\in A'$, by (\ref{e2.1}) we have $\frac{\overline{b_{i}}+(b_{i+1}-b_{i})}{2}\not\in A'$.

Similar to the above discuss, we have
$$\left|\left[0, \overline{b_{i}}+(b_{i+1}-b_{i})\right]\cap A'\right|=\frac{\overline{b_{i}}+(b_{i+1}-b_{i})}{2}, $$
thus $a_{\frac{\overline{b_{i}}}{2}+\frac{b_{i+1}-b_{i}}{2}}>\overline{b_{i}}+(b_{i+1}-b_{i})$, a contradiction.

Hence, $b_{i+1}\geqslant 2b_{i}+2$.

This completes the proof of Lemma \ref{L2-3}.
\end{proof}

\begin{lem}\label{P1}
Let $m\geqslant 2$. Then there are no consecutive integers in $[2k-3, 2k-4+b_{m-1}]\backslash 2^{\wedge}A'$ unless $m=2$ and
$$[0, b_{2}]\cap A'=\left[0,\frac{b_{1}}{2}\right]\cup \left[b_{1}+1, \frac{3b_{1}}{2}+1\right]. $$
In this case, we have $|2^{\wedge}A|\geqslant 3k-6$.
\end{lem}

\begin{lem}\label{P2}
Let $m\geqslant 2$. Then there are no integers with a difference of $2$ in $[2k-3, 2k-4+b_{m-1}]\backslash 2^{\wedge}A'$.
\end{lem}

\begin{lem}\label{P3}
Let $m\geqslant 2$. Then there are no integers with a difference of $3$ in $[2k-3, 2k-4+b_{m-1}]\backslash 2^{\wedge}A'$ unless $m=3$, $b_{1}=2$ and one of the following cases holds:

$(i)$ $b_{2}\equiv 2\pmod{3}$ and
\begin{eqnarray*}
[0, b_{3}]\cap A'
&=&\left\{3i: 0\leqslant i\leqslant \frac{2b_{2}+2}{3}\right\}\\
&&\cup \left\{3i+1: 0\leqslant i\leqslant \frac{b_{2}-2}{6}, \frac{b_{2}+1}{3}\leqslant i\leqslant \frac{b_{2}}{2}\right\};
\end{eqnarray*}

$(ii)$ $b_{2}\equiv 0\pmod{3}$ and
\begin{eqnarray*}
[0, b_{3}]\cap A'
&=&\left\{3i: 0\leqslant i\leqslant \frac{b_{2}}{6}\right\}\cup \left\{3i+1: 0\leqslant i\leqslant \frac{b_{2}}{2}\right\}\\
&&\cup \left\{b_{2}+2+3i: 0\leqslant i\leqslant \frac{b_{2}}{3}\right\};
\end{eqnarray*}

$(iii)$ $b_{2}\equiv 2\pmod{3}$ and
\begin{eqnarray*}
[0, b_{3}]\cap A'
&=&\left\{3i: 0\leqslant i\leqslant \frac{b_{2}}{2}+1\right\}\cup \left\{3i+1: 0\leqslant i\leqslant \frac{b_{2}-2}{6}\right\}\\
&&\cup \left\{b_{2}+3+3i: 0\leqslant i\leqslant \frac{b_{2}+1}{3}\right\};
\end{eqnarray*}

$(iv)$ $b_{2}\equiv 1\pmod{3}$ and
\begin{eqnarray*}
[0, b_{3}]\cap A'
&=&\left\{3i: 0\leqslant i\leqslant \frac{b_{2}+2}{6}\right\}\cup \left\{b_{2}+3+3i: 0\leqslant i\leqslant \frac{b_{2}+2}{3}\right\}\\
&&\cup \left\{3i+2: 0\leqslant i\leqslant \frac{b_{2}}{2}\right\};
\end{eqnarray*}

$(v)$ $b_{2}\equiv 0\pmod{3}$ and
\begin{eqnarray*}
[0, b_{3}]\cap A'
&=&\left\{3i: 0\leqslant i\leqslant \frac{b_{2}}{6}, \ \ \frac{b_{2}}{3}+1\leqslant i\leqslant \frac{b_{2}}{2}+1\right\}\\
&&\cup \left\{3i+1: 0\leqslant i\leqslant \frac{2b_{2}}{3}+1\right\};
\end{eqnarray*}

$(vi)$ $b_{2}\equiv 2\pmod{3}$ and
\begin{eqnarray*}
[0, b_{3}]\cap A'&=&\left\{3i: 0\leqslant i\leqslant \frac{2b_{2}+5}{3}\right\}\\
&&\cup \left\{3i+2: 0\leqslant i\leqslant \frac{b_{2}-2}{6}, \ \ \frac{b_{2}+1}{3}\leqslant i\leqslant \frac{b_{2}}{2}\right\}.
\end{eqnarray*}

In all above cases, we have $|2^{\wedge}A|\geqslant 3k-6$.
\end{lem}

\begin{lem}\label{L1-5}
Let $A$ be as in Theorem \ref{T2}. Write
$$D=\{d\in [1, b_{m-1}]: 2k-4+d\in 2^{\wedge}A'\}. $$
Then
$$|D|\geqslant \frac{b_{m-1}}{2}+\left\lfloor\frac{b_{m-1}}{4}\right\rfloor, $$
except for finitely exceptions. In those exceptions, we have $|2^{\wedge}A|\geqslant 3k-6$.
\end{lem}

\begin{proof}
Write
$$C=\{c\in [1, b_{m-1}]: 2k-4+c\not\in 2^{\wedge}A'\}. $$
Then
$$C\cup D=[1, b_{m-1}], \ \ C\cap D=\emptyset. $$

Let $c_{1}, c_{2}\in C$.
By Lemmas \ref{P1}-\ref{P3}, we have $c_{2}-c_{1}\geqslant 4$,
except for finite exceptional sets $A$ given in Lemmas \ref{P1} and \ref{P3}.
In those exceptions, we have $|2^{\wedge}A|\geqslant 3k-6$.

By the division algorithm, there exist integers $q$ and $r$ such that
$$b_{m-1}=4q+r, \ \ r=0 \text{ or } 2. $$
If $r=0$, then
$$|[1, b_{m-1}]\cap C|=\left|\bigcup_{i=0}^{q-1}[4i, 4(i+1)]\cap C\right|\leqslant \sum_{i=0}^{q-1}1=q. $$
Hence
$$|D|=b_{m-1}-|C|\geqslant b_{m-1}-q=3q. $$
If $r=2$, then
\begin{eqnarray*}
|[1, b_{m-1}]\cap C|&=&\left|\bigcup_{i=0}^{q-1}[4i, 4(i+1)]\cap C\right|+|\{4q+1, 4q+2\}\cap C|\\
&\leqslant& 1+\sum_{i=0}^{q-1}1=q+1.
\end{eqnarray*}
Hence
$$|D|=b_{m-1}-|C|\geqslant b_{m-1}-(q+1)=3q+1. $$

In all,
$$|D|\geqslant \frac{b_{m-1}}{2}+\left\lfloor\frac{b_{m-1}}{4}\right\rfloor. $$

This completes the proof of Lemma \ref{L1-5}.
\end{proof}

\begin{pro}\label{P4}
Let $m\geqslant 2$. Then $2k-3+b_{m-1}$, $2k-2+b_{m-1}\not\in 2^{\wedge}A'$ if and only if $m=2$ and one of the following cases holds:

(i) $k\equiv 1\pmod {2}$ and
$$A'=\left[0, \frac{k-3}{2}\right]\cup \left[k-2, \frac{3(k-3)}{2}+1\right]$$ and $B=\{k-3, 2k-4\}$;

(ii) $k\equiv 0\pmod {3}$ and
$$A'=\left[0, \frac{k-3}{3}\right]\cup\left[\frac{2k-3}{3}, k-2\right]\cup\left[\frac{4k-6}{3}, \frac{5k-12}{3}\right]$$ and $B=\left\{\frac{2k-6}{3}, 2k-4\right\}$;

(iii) $k\equiv 1\pmod {3}$ and
$$A'=\left[0, \frac{k-4}{3}\right]\cup \left[\frac{2k-5}{3}, k-3\right]\cup \left[\frac{4k-7}{3}, \frac{5k-11}{3}\right]$$ and 
$B=\left\{\frac{2k-8}{3}, \frac{4k-10}{3}\right\}$.
\end{pro}

\section{Proof of Theorem \ref{T2}}\label{S5}
It is easy to verify that the Freiman-Lev conjecture is true for all $3\leqslant k\leqslant 9$, so we assume that $k>9$.
Let $A'=A\setminus\{a_{k-1}\}$ and $B=[1, 2k-4]\setminus 2^{\wedge}A'$. Let $m$ be a positive integer and
$$B=\{b_{1}<b_{2}<\cdots<b_{m}\}, \ \ b_{0}=0. $$
We divide into the following three cases:

{\bf Case 1. }$m\leqslant1$. Then
$$[1, 2k-4]\backslash B\subseteq 2^{\wedge}A'\subseteq 2^{\wedge}A, \ \ a_{k-1}+A'\subseteq 2^{\wedge}A$$
and
$$([1, 2k-4]\backslash B)\cap (a_{k-1}+A')=\emptyset, $$
thus
$$|2^{\wedge}A|\geqslant (2k-4-1)+(k-1)=3k-6. $$

{\bf Case 2. }$m\geqslant 3$. Then our goal is find enough many elements of $2^{\wedge}A$.

\begin{itemize}
  \item Firstly, the set $[1, 2k-4]\backslash B$ provides $2k-4-m$ integers of $2^{\wedge}A'$.

  \item By Lemma \ref{L2-2} (ii), we have
  $$\left|[0, b_{m-1}]\cap A'\right|=\frac{b_{m-1}}{2}+1, $$
  thus
  $$|[b_{m-1}+1, a_{k-2}]\cap A'|=k-2-\frac{b_{m-1}}{2}. $$
  Hence, $a_{k-1}+([b_{m-1}+1, a_{k-2}]\cap A)$ provides $k-2-\frac{b_{m-1}}{2}$ integers of $2^{\wedge}A$.

  \item By Proposition \ref{P4}, one of $2k-3+b_{m-1}$ and $2k-2+b_{m-1}$ belongs to $2^{\wedge}A'$.
\end{itemize}
To sum up, we obtain at least $3k-6-\frac{b_{m-1}}{2}-m+1$ integers of $2^{\wedge}A$.

If $m\geqslant 4$, then by Lemma \ref{L2-3}, we have
$$b_{m-1}\geqslant 2b_{m-2}+2\geqslant \cdots\geqslant \sum_{i=1}^{m-1}2^{i}=2^{m}-2. $$
By Lemma \ref{L1-5}, we have
$$|D|\geqslant \frac{b_{m-1}}{2}+\left\lfloor\frac{2^{m}-2}{4}\right\rfloor\geqslant \frac{b_{m-1}}{2}+m-1. $$
Hence
$$|2^{\wedge}A|\geqslant 3k-6. $$

Assume that $m=3$. If $b_{2}\geqslant 8$, then $|D|\geqslant \frac{b_{2}}{2}+2$, thus $|2^{\wedge}A|\geqslant 3k-6$.
Let $b_{1}=2$ and $b_{2}=6$. To prove $|2^{\wedge}A|\geqslant 3k-6$, we shall show that $|D|\geqslant 5$.

Suppose, for a contradiction, that $|D|=4$. By Lemma \ref{L2-2} (ii), we have
$$[1, 6]\cap A'=\{1, 3, 4\}. $$
By Lemma \ref{L1-5}, we have
$$D=\{2, 3, 4, 6\}, \ \ \{2, 3, 4, 5\} \text{ or } \{1, 3, 4, 5\}. $$
If $D=\{2, 3, 4, 6\}$, then
$$2k-3, 2k+1\not\in 2^{\wedge}A'. $$
By Lemma \ref{L2-2} (iii), we have
$$|\{i, 2k-3-i\}\cap A'|=1, \ \ i=7, \ldots, k-2, $$
$$|\{i, 2k+1-i\}\cap A'|=1, \ \ i=7, \ldots, k. $$
If $b_{3}<2k-4$, then $b_{3}+1\in A'$, thus $2k+1-(b_{3}+1)\not\in A'$,
and so
$$2k-3-(2k+1-(b_{3}+1))=b_{3}-3\in A', $$
which contradicts with $3\in A'$.
Hence, $b_{3}=2k-4$, and so
$$k-2\in A'\Leftrightarrow k-1\not\in A'\Leftrightarrow k-3\in A'\Leftrightarrow k\not\in A'\Leftrightarrow k+1, k-4\in A', $$
which contradicts with $2k-3\not\in 2^{\wedge}A'$.

The cases $D=\{1, 3, 4, 5\}$ and $D=\{2, 3, 4, 5\}$ are similar to the above.

In all, $|2^{\wedge}A|\geqslant 3k-6$.

{\bf Case 3. }$m=2$.
Besides the above case, we have the following $3k-8-\frac{b_{1}}{2}$ elements of $2^{\wedge}A$:
$$[1, 2k-4]\backslash B, \ \ a_{k-1}+[b_{1}+1, a_{k-2}]\cap A'. $$
If $b_{1}\geqslant 8$, then by Lemma \ref{L1-5}, we have $|D|\geqslant \frac{b_{1}}{2}+2$, thus $|2^{\wedge}A|\geqslant 3k-6$.

{\bf Subcase 3.1. }$b_{1}=6$. Then $|D|\geqslant \frac{b_{1}}{2}+1$.
By Lemma \ref{L2-2} (i) and (ii), we have
$$[1, 6]\cap A'=\{1, 2, 3\}. $$
If $\{2k+3, 2k+4\}\cap 2^{\wedge}A'\neq\emptyset$, then $|2^{\wedge}A|\geqslant 3k-6$.
Assume that
$$\{2k+3, 2k+4\}\cap 2^{\wedge}A'=\emptyset. $$
By Proposition \ref{P4}, we have
$$A'=[0, 3]\cup [7, 10], \ \ [0, 3]\cup [7, 10]\cup [14, 16], \ \ [0, 3]\cup [7, 10]\cup [15, 18], $$
it implies that $|2^{\wedge}A|\geqslant 3k-6$.

{\bf Subcase 3.2. }$b_{1}=4$. Then $|2^{\wedge}A|\geqslant 3k-6$. The proof is similar to subcase 3.1.

{\bf Subcase 3.3. }$b_{1}=2$. By Lemma \ref{L2-2} (i) and (iv), we have
$$1\in A', \ \ 2\not\in A', \ \ 3\in A. $$
By Lemma \ref{L1-5},
$$\{2k-3, 2k-2\}\cap 2^{\wedge}A'\neq \emptyset. $$
Since $a_{k-1}\geqslant 2k-2$, we have
$$[1, 2k-4]\backslash\{2, b_{2}\}\subseteq 2^{\wedge}A', \ \ a_{k-1}+A'\subseteq 2^{\wedge}A. $$
If $a_{k-1}>2k-2$, then $|2^{\wedge}A|\geqslant 3k-6$.
Assume that $a_{k-1}=2k-2$. If $2k-3\in 2^{\wedge}A'$ or $2k\in 2^{\wedge}A'$, then $|2^{\wedge}A|\geqslant 3k-6$.

Assume that $a_{k-1}=2k-2$ and $2k-3$, $2k\not\in 2^{\wedge}A'$.
Similar to Lemma \ref{P3}, we have
$$A=\left\{3i, 3i+1, k+1+3i: i=0, \ldots, \frac{k-3}{3}\right\}, \ \ k\equiv 0\pmod{3}$$
or
$$A=\left\{3i, 3i+1, k-1+3i, 2k-2: i=0, \ldots, \frac{k-4}{3}\right\}, \ \ k\equiv 1\pmod{3}, $$
thus $b_{2}=2k-6$ and $|2^{\wedge}A|=3k-7$.

This completes the proof of Theorem \ref{T2}.

\section{Preliminaries of Theorem \ref{T3}}
Let $A=\{0=a_{0}<a_{1}<\cdots<a_{k-1}\}$. For any integer $w$, let $S(w)=\{w, w+a_{k-1}\}$ and 
\begin{equation}\label{e2-01}
W=\left\{w\in [0, a_{k-1}]\backslash A:  S(w)\cap 2^{\wedge}A=\emptyset\right\}.
\end{equation}

\begin{lem}\label{L3-1}(\cite{Wang}, Theorem 1.1)
Let $A$ be a set of $k\geqslant 5$ integers such that $A\subseteq [0, l]$, $0, l\in A$ and $\gcd(A)=1$. If $l\leqslant 2k-3$, then $|W|\leqslant 2$.
\end{lem}

Assume that $W=\{w_{1}, w_{2}\}$ and $\gcd(w_{2}-w_{1}, a_{k-1})=m$. Write
\begin{equation}\label{e2-02}V=\left\{\frac{w_{2}}{2}, \frac{w_{2}+a_{k-1}}{2}\right\}\cap \mathbb{Z}, \ \ H=m\mathbb{Z}\cap [0,a_{k-1}). \end{equation}
For $v\in V$, $x\in\left[0, \frac{a_{k-1}}{m}-1\right]$, let
\begin{equation}\label{e2-03}r_v(x):=v+x(w_{2}-w_{1})-q(x)a_{k-1}, \end{equation}
where $q(x)$ is the unique integer such that $0\leqslant r_v(x)<a_{k-1}$.
Define
$$\mathcal{D}^{-}(v):=\left\{r_v(x): x\in\left[0, \frac{a_{k-1}}{2m}-\frac{1}{2}\right]\right\}. $$

Actually, if $\frac{w_{2}-w_{1}}{m}$ is even, then
\begin{equation}\label{eq3.2-1}\frac{w_{1}}{2}=\frac{w_{2}}{2}+\left(\frac{l}{2m}-\frac{1}{2}\right)(w_{2}-w_{1})-\frac{w_{2}-w_{1}}{2m}l, \end{equation}
\begin{equation}\label{eq3.2-2}\frac{w_{1}+l}{2}=\frac{w_{2}+l}{2}+\left(\frac{l}{2m}-\frac{1}{2}\right)(w_{2}-w_{1})-\frac{w_{2}-w_{1}}{2m}l. \end{equation}
If $\frac{w_{2}-w_{1}}{m}$ is odd, then
\begin{equation}\label{eq3.2-3}\frac{w_{1}+l}{2}=\frac{w_{2}}{2}+\left(\frac{l}{2m}-\frac{1}{2}\right)(w_{2}-w_{1})-\left(\frac{w_{2}-w_{1}}{2m}-\frac{1}{2}\right)l, \end{equation}
\begin{equation}\label{eq3.2-4}\frac{w_{1}}{2}=\frac{w_{2}+l}{2}+\left(\frac{l}{2m}-\frac{1}{2}\right)(w_{2}-w_{1})-\left(\frac{w_{2}-w_{1}}{2m}+\frac{1}{2}\right)l. \end{equation}
By the above (\ref{eq3.2-1})-(\ref{eq3.2-4}), we know that $\frac{w_{1}}{2}$ and $\frac{w_{1}+l}{2}$ can be represented as the form $r(x)$.
These equations will be used later.

\begin{lem}\label{L3-2}(\cite{Wang}, Proposition 1.4)
Let $A$ be a set of $k>7$ integers such that $A\subseteq [0, l]$, $0, l\in A$ and $l\leqslant 2k-3$.
Let $w_{1}, w_{2}\in W$ and $\gcd(w_{2}-w_{1}, a_{k-1})=m$. Then
$$A=\{a_{k-1}\}\cup(U+H)\cup\bigcup_{v\in V}\mathcal{D}^{-}(v), $$
where $U=\{u\in [0, m-1]\cap A: 2u\not\equiv w_{2}\pmod{m}\}$.
\end{lem}

\begin{lem}\label{L3-3}
Let $k>7$. Let $A=\{0=a_{0}<a_{1}<\cdots<a_{k-1}\}$ be the set of integers.
Let $w_{1}, w_{2}\in W$ and $\gcd(w_{2}-w_{1}, a_{k-1})=m$. Let $V$, $H$, $U$ and $\mathcal{D}^{-}(v)$ be the sets defined as above.
If $a_{k-1}=2k-3$, then

(i) $|U|=\frac{m-1}{2}$;

(ii) Let $u_{1}, u_{2}\in [0, m-1]$ be two different integers such that $u_{1}+u_{2}\equiv w_{2}\pmod{m}$.
Then $|\{u_{1}, u_{2}\}\cap U|=1$.
\end{lem}

\begin{proof}
Since $a_{k-1}=2k-3$, we know that $m=\gcd(w_{2}-w_{1}, a_{k-1})$ is odd.
By the definition of $V$, we have $|V|=1$.
Put $V=\{v\}$. By Lemma \ref{L3-2}, we have
\begin{equation}\label{e4-1}A=\{a_{k-1}\}\cup(U+H)\cup\mathcal{D}^{-}(v). \end{equation}
Let $\overline{x}$ be the least nonnegative residue of $x$ modulo $m$.

To prove (i). By the definitions of $\mathcal{D}^{-}(v)$ and $H$, we have
$$\left|\mathcal{D}^{-}(v)\right|=\frac{a_{k-1}}{2m}+\frac{1}{2}, \ \ |H|=\frac{a_{k-1}}{m}. $$
Since $|A|=k$, by (\ref{e4-1}) we have
$$|U|=\frac{m-1}{2}. $$

To prove (ii). If $u_{1}, u_{2}\in U$, then by (\ref{e4-1}) we have
$$u_{1}+H, \ \ u_{2}+H\subseteq A, $$
thus $w_{2}\in 2^{\wedge}A$, a contradiction.
Hence, $|\{u_{1}, u_{2}\}\cap U|\leqslant 1$, it implies that $|U|\leqslant \frac{m-1}{2}$.
By (i) we have $$|\{u_{1}, u_{2}\}\cap U|=1. $$

This completes the proof of Lemma \ref{L3-3}.
\end{proof}

\section{Proof of Theorem \ref{T3}}\label{S6}
It is easy to verify that the conclusion is valid for $3\leqslant k<15$. Now, we consider $k\geqslant 15$.
Let $W$, $V$ and $r_{v}(x)$ be as in (\ref{e2-01})-(\ref{e2-03}).

Write
$$T=\{a_{i}: 1\leqslant i\leqslant k-2\}\cup\{a_{i}+a_{k-1}: 0\leqslant i\leqslant k-2\}. $$
Then $T\subseteq 2^{\wedge}A$ and $|T|=2k-3$. By Lemma \ref{L3-1}, we have $|W|\leqslant 2$, thus
$$|S(w)\cap 2^{\wedge}A|\geqslant 1, \ \ w\in [0, a_{k-1}]\backslash(A\cup W). $$
Hence,
$$|2^{\wedge}A|\geqslant |T|+|[0, a_{k-1}]\backslash(A\cup W)|\geqslant 3k-7. $$
Therefore, $|2^{\wedge}A|=3k-7$ if and only if $|W|=2$ and
\begin{equation}\label{e}|S(w)\cap 2^{\wedge}A|=1, \ \ w\in [0, a_{k-1}]\backslash(A\cup W). \end{equation}
Put $V=\{v\}$. By Lemma \ref{L3-2}, we have
\begin{equation}\label{e4-002}A=\{a_{k-1}\}\cup(U+H)\cup\mathcal{D}^{-}(v). \end{equation}

Firstly, we show that if $k\geqslant 15$ and $|2^{\wedge}A|=3k-7$, then $m<5$.

Suppose that $m\geqslant 5$. By Lemma \ref{L3-3} (i), we have $|U|\geqslant 2$.
Let $\alpha\in [0, m-1]\backslash U$ be an integer such that $\alpha\not\equiv v\pmod{m}$ and $\alpha\not\equiv 2v\pmod{m}$.
Then $\alpha\not\in A\cup W$. By (\ref{e4-002}) we have $(\alpha+H)\cap A=\emptyset$, thus for any  $h\in H$, we have
\begin{equation}\label{e4-2}|S(\alpha+h)\cap 2^{\wedge}A|=1. \end{equation}
Hence, one of the following three possibilities holds:
$$|S(\alpha)\cap 2^{\wedge}\mathcal{D}^{-}(v)|=1, \ \ |S(\alpha)\cap ((U+H)+\mathcal{D}^{-}(v))|=1, \ \ |S(\alpha)\cap 2^{\wedge}(U+H)|=1. $$

Since $\alpha\not\equiv 2v\pmod{m}$, we have $S(\alpha)\cap 2^{\wedge}\mathcal{D}^{-}(v)=\emptyset$.
Now, we shall show that
$$S(\alpha)\cap ((U+H)+\mathcal{D}^{-}(v))=\emptyset. $$
If not, then there exists $u\in U$ such that
$$\alpha\equiv u+v\pmod{m}. $$
Since $\alpha\not\in U$, we have $v\not\equiv 0\pmod{m}$, thus $0\not\in\mathcal{D}^{-}(v)$.
By the definition of $r_{v}(x)$, we have
$$q(x)=q(x-1) \text{ or } q(x-1)+1$$
for any $x\in [1, \frac{a_{k-1}}{2m}-\frac{1}{2}]$.
Clearly, $q(x)=q(x-1)$ for some $x\in [1, \frac{a_{k-1}}{2m}-\frac{1}{2}]$ implies that there exist $v_{1}, v_{2}\in \mathcal{D}^{-}(v)$ such that $v_{2}=v_{1}+(w_{2}-w_{1})$,
then for some $h\in H$, we have
$$\alpha+h=u+v_{1}\in 2^{\wedge}A, $$
$$\alpha+h+a_{k-1}=u+v_{1}+a_{k-1}=(u+(a_{k-1}-(w_{2}-w_{1})))+v_{2}\in 2^{\wedge}A. $$
Thus $|S(\alpha+h)\cap 2^{\wedge}A|=2$, which contradicts with (\ref{e4-2}).
Hence, $q(x)=q(x-1)+1$ for any $x\in [1, \frac{a_{k-1}}{2m}-\frac{1}{2}]$, so
\begin{equation}\label{e22}
\mathcal{D}^{-}(v)=\left\{v+x(w_{2}-w_{1})-xa_{k-1}, \ \ 0\leqslant x\leqslant \frac{a_{k-1}}{2m}-\frac{1}{2}\right\}.
\end{equation}
Let $c\geqslant 1$ be the integer such that $a_{k-1}-(w_{2}-w_{1})=cm$. Then
$$u+(v+(w_{2}-w_{1})-a_{k-1})=u+v-cm\in 2^{\wedge}A, $$
$$\left(u+a_{k-1}-cm\right)+v=u+v-cm+a_{k-1}\in 2^{\wedge}A. $$
Thus $|S(u+v-cm)\cap 2^{\wedge}A|=2$, which also contradicts with (\ref{e4-2}).
Hence,
$$S(\alpha)\cap ((U+H)+\mathcal{D}^{-}(v))=\emptyset, $$
so
$$|S(\alpha)\cap 2^{\wedge}(U+H)|=1. $$

Let $u_{1}, u_{2}\in U$ be two integers such that
$$\alpha\equiv u_{1}+u_{2}\pmod{m}. $$

Assume that $u_{1}\neq u_{2}$. Then $\alpha\in 2^{\wedge}A$ and
$$\alpha+a_{k-1}=(u_{1}+m)+(u_{2}+a_{k-1}-m)\in 2^{\wedge}(U+H)\subseteq 2^{\wedge}A, $$
thus $|S(\alpha)\cap 2^{\wedge}A|=2$, which contradicts with (\ref{e4-2}). Hence, $u_{1}=u_{2}$.

Assume that $a_{k-1}>3m$. Then $|H|\geqslant 4$.
Similar to the above discuss, we have
$$|S(\alpha+m)\cap 2^{\wedge}A|=2, $$
which also contradicts with (\ref{e4-2}).
Hence, $a_{k-1}=3m$.

Assume that $0\not\in U$. Since $0\in A$, by (\ref{e4-002}) we have $0\in \mathcal{D}^{-}(v)$,
thus $v\equiv 0\pmod{m}$, it implies that $w_{1}\equiv w_{2}\equiv 0\pmod{m}$.
By the definition of $\mathcal{D}^{-}(v)$, we have $|\mathcal{D}^{-}(v)|=2$, thus $\mathcal{D}^{-}(v)=\{v, 0\}$.
Hence, $a_{k-1}=v+(w_{2}-w_{1})$,
it follows that
$$\{w_{1}, w_{2}\}=\{0, m\}, \ \ \{0, 2m\} \text{ or } \{m, 2m\}, $$
which contradicts with the definition of $W$.
Therefore, $0\in U$ and $v\not\equiv 0\pmod{m}$.

To sum up, if $|2^{\wedge}A|=3k-7$ and $m\geqslant 5$, then the following facts are hold:

{\bf Fact A} $a_{k-1}=3m$;

{\bf Fact B} Let $\alpha\in [0, m-1]\backslash U$ and $\alpha\not\equiv v\pmod{m}$, $\alpha\not\equiv 2v\pmod{m}$.
Then $\alpha\equiv 2u\pmod{m}$ for some $u\in U$, $\alpha\not\equiv v+u\pmod{m}$ for any $u\in U$
and $\alpha\not\equiv u_{1}+u_{2}\pmod{m}$ for any two distinct integers $u_{1}, u_{2}\in U$.

{\bf Fact C} $0\in U$ and $v\not\equiv 0\pmod{m}$.

Write
$$U=\left\{0=u_{0}<u_{1}<\cdots<u_{\frac{m-3}{2}}\right\}. $$
Since we assume that $k\geqslant 15$, we have $a_{k-1}\geqslant 27$, thus $m\geqslant 9$. Hence, $|U|\geqslant 4$.

Let $\overline{x}$ be the least nonnegative residue of $x$ modulo $m$. Then
\begin{equation}\label{e4-1-2}[0, m-1]=\left\{u_{i}, \overline{2u_{i}}: 1\leqslant i\leqslant \frac{m-3}{2}\right\}\cup\left\{0, \overline{v}, \overline{2v}\right\}, \end{equation}
it follows that
\begin{equation}\label{e4-3}\overline{2u_{i}}\not\in U, \ \ i=1, \ldots, \frac{m-3}{2}. \end{equation}
By the Fact B and Lemma \ref{L2-3} (ii), we have
\begin{equation}\label{e4-1-3}\overline{u_{1}+u_{i}}\in U\cup\{\overline{v}\}, \ \ i=2, \ldots, \frac{m-3}{2}, \end{equation}
\begin{equation}\label{e4-1-1}\overline{v+u_{i}}\in U, \ \ i=1, \ldots, \frac{m-3}{2}. \end{equation}

Now, we show that $u_{1}>2$.

Firstly, we assume that $u_{1}=1$. Noting that
$$U=\left\{0=u_{0}<u_{1}<\cdots<u_{\frac{m-3}{2}}\right\}\subseteq [0, m-1], $$
then
$$\overline{1+u_{i}}=1+u_{i}, \ \ 2\leqslant i\leqslant \frac{m-5}{2}$$
and
$$\overline{1+u_{\frac{m-3}{2}}}=\left\{
\begin{array}{ll}
1+u_{\frac{m-3}{2}} & \text{ if } \frac{m-3}{2}<m-1, \\
0 & \text{ if } \frac{m-3}{2}=m-1.
\end{array}\right.
$$
It is clear that
$$u_{2}<1+u_{2}<1+u_{3}<\cdots<1+u_{\frac{m-5}{2}}<1+u_{\frac{m-3}{2}}. $$
If $\overline{1+u_{i}}\in U$ for all $2\leqslant i\leqslant \frac{m-3}{2}$, then
$$\overline{1+u_{i}}=1+u_{i}=u_{i+1}, \ \ i=2, \ldots, \frac{m-5}{2}$$
and $\overline{1+u_{\frac{m-3}{2}}}=0$. Thus $u_{\frac{m-3}{2}}=m-1$, and then $u_{\frac{m-5}{2}}=m-2$.
By (\ref{e4-3}) we have
$$\overline{2u_{\frac{m-3}{2}}}=m-2\not\in U, $$
a contradiction. Hence, there exists an integer $2\leqslant t\leqslant \frac{m-3}{2}$ such that $\overline{1+u_{t}}\not\in U$.
By (\ref{e4-1-3}) we have $\overline{1+u_{t}}=\overline{v}$. We will discuss three separate situations as follows:
\begin{itemize}
  \item If $2\leqslant t\leqslant \frac{m-7}{2}$, then
$$\overline{1+u_{i}}=1+u_{i}=u_{i+1}, \ \ i=t+1, \ldots, \frac{m-5}{2}$$
and $\overline{1+u_{\frac{m-3}{2}}}=0$, thus $u_{\frac{m-3}{2}}=m-1$ and $u_{\frac{m-5}{2}}=m-2$.
By (\ref{e4-3}) we have
$$\overline{2u_{\frac{m-3}{2}}}=m-2\not\in U, $$
a contradiction.
  \item If $\overline{1+u_{\frac{m-5}{2}}}=\overline{v}$, then $\overline{1+u_{\frac{m-3}{2}}}=0$, thus
$$U=\left\{0, 1, u_{2}, u_{2}+1, \ldots, \overline{v-1}, m-1\right\}. $$
By (\ref{e4-1-1}) we have $\overline{v+1}\in U$, thus $\overline{v+1}=m-1$, it follows that $\overline{2v}=m-4$.
Since $\overline{2v}\not\in U$, we have $u_{2}\geqslant m-3$, thus $U=\left\{0, 1, m-3, m-1\right\}$,
it follows that $m=9$ and $U=\left\{0, 1, 6, 8\right\}$. Clearly, $7=1+6\not\in U$, a contradiction.
  \item If $\overline{1+u_{\frac{m-3}{2}}}=\overline{v}$, then
$$U=\left\{0, 1, u_{2}, \ldots, \overline{v-1}\right\}. $$
By (\ref{e4-1-1}) we have $\overline{v+1}\in U$, thus $\overline{v+1}=0$, it follows that $\overline{2v}=m-2$, which contradicts with $\overline{v-1}=m-2\in U$.
\end{itemize}

Assume that $u_{1}=2$. Then $1\not\in U$. The proof is similar to the above, we omit it.

Hence, $1, 2\not\in U$, so $\overline{2u_{i}}\neq 2$ for all $2\leqslant t\leqslant \frac{m-3}{2}$.
By (\ref{e4-1-2}) we have $\overline{v}=1$ or $2$. Based on (\ref{e4-1-1}) and similar discussions of the above, we can obtain a contradiction.

In all, if $k\geqslant 15$ and $|2^{\wedge}A|=3k-7$, then $m<5$. We divide into two cases:

{\bf Case 1. }$m=1$. Then $U=\emptyset$. By (\ref{e4-002}) we have
$$A=\{a_{k-1}\}\cup\mathcal{D}^{-}(v), \ \ 0\in \mathcal{D}^{-}(v). $$
For any non-negative integer $i$, define
$${\bf v}(i, q_{i})=i(w_{2}-w_{1})-q_{i}a_{k-1}, $$
where $q_{i}$ is the unique integer such that ${\bf v}(i, q_{i})\in[0, a_{k-1})$.
Since $\gcd(w_{2}-w_{1}, a_{k-1})=1$, we have
$$[0, a_{k-1}-1]=\{{\bf v}(i, q_{i}): 0\leqslant i\leqslant a_{k-1}-1\}. $$
Noting that there exists an integer $0\leqslant t\leqslant 2k-4$ such that
$$r_{v}(0)=v={\bf v}(t, q_{t}). $$
By the definition of $\mathcal{D}^{-}(v)$, we have
$$r_{v}(1)={\bf v}(t+1, q_{t+1}), \ldots, r_{v}(2k-4-t)={\bf v}(2k-4, q_{2k-4}), $$
$$0=r_{v}(2k-3-t)={\bf v}(0, q_{0}), \ldots, r_{v}(k-2)={\bf v}(t-k+1, q_{t-k+1}), $$
that is,
\begin{equation}\label{e4-7}\mathcal{D}^{-}(v)=\left\{{\bf v}(i, q_{i}): i\in [0, t-k+1]\cup[t, 2k-4]\right\}. \end{equation}
Hence
\begin{equation}\label{e4-6}\{{\bf v}(i, q_{i}): i\in [t-k+2, t-1]\}\cap A=\emptyset. \end{equation}

Since $k\geqslant 15$, we have $|\mathcal{D}^{-}(v)|\geqslant 14$, thus
$$|\{{\bf v}(i, q_{i}): i\in [1, t-k+1]\}|\geqslant 7 \text{ or } |\{{\bf v}(i, q_{i}): i\in [t, 2k-4]\}|\geqslant 7. $$

Without loss of generality, we may assume that
\begin{equation}\label{e4-8}|\{{\bf v}(i, q_{i}): i\in [1, t-k+1]\}|\geqslant |\{{\bf v}(i, q_{i}): i\in [t, 2k-4]\}|. \end{equation}
Then
$$|\{{\bf v}(i, q_{i}): i\in [1, t-k+1]\}|\geqslant 7, $$
thus
$$t-k+1\geqslant 7. $$

Assume that $q_{t-k}=q_{t-k+1}$.
For any $1\leqslant j\leqslant t-k-2$, we have
$$0<{\bf v}(t-k, q_{t-k})+{\bf v}(j+1, q_{j+1})<2a_{k-1}-1, $$
thus
\begin{equation}\label{e4-2-1}{\bf v}(t-k+1+j, q_{t-k+1+j})={\bf v}(t-k, q_{t-k})+{\bf v}(j+1, q_{j+1})\end{equation}
or
\begin{equation}\label{e4-2-2}{\bf v}(t-k+1+j, q_{t-k+1+j})={\bf v}(t-k, q_{t-k})+{\bf v}(j+1, q_{j+1})-a_{k-1}. \end{equation}

Next, we show that $q_{j+1}=q_{j}$ for all $j\in [1, t-k-2]$.

Assume that there exists an integer $j\in [1, t-k-2]$ such that $q_{j+1}=q_{j}+1$.
If (\ref{e4-2-1}) holds, then
$${\bf v}(t-k+1+j, q_{t-k+1+j})={\bf v}(t-k, q_{t-k})+{\bf v}(j+1, q_{j+1})\in 2^{\wedge}A, $$
$${\bf v}(t-k+1+j, q_{t-k+1+j})+a_{k-1}={\bf v}(t-k+1, q_{t-k+1})+{\bf v}(j, q_{j})\in 2^{\wedge}A, $$
it implies that
$$|S({\bf v}(t-k+1+j, q_{t-k+1+j}))\cap 2^{\wedge}A|=2. $$
If (\ref{e4-2-2}) holds, then
$${\bf v}(t-k+1+j, q_{t-k+1+j})={\bf v}(t-k+1, q_{t-k+1})+{\bf v}(j, q_{j})\in 2^{\wedge}A, $$
$${\bf v}(t-k+1+j, q_{t-k+1+j})+a_{k-1}={\bf v}(t-k, q_{t-k})+{\bf v}(j+1, q_{j+1})\in 2^{\wedge}A, $$
it follows that
$$|S({\bf v}(t-k+1+j, q_{t-k+1+j}))\cap 2^{\wedge}A|=2. $$
So ${\bf v}(t-k+1+j, q_{t-k+1+j})\not\in W$. Since
$$t-k+2\leqslant t-k+1+j\leqslant 2t-2k-1<t-1, $$
by (\ref{e4-6}) we have ${\bf v}(t-k+1+j, q_{t-k+1+j})\not\in A$, which contradicts with (\ref{e}).
By (\ref{e4-7}) we have
$$q_{1}=\cdots =q_{t-k-1}=0. $$

Now, we prove that $q_{t-k}=q_{t-k+1}=0$. If not, $q_{t-k}=q_{t-k+1}=1$, then
$$\left\{
\begin{array}{lll}
{\bf v}(t-k-4, 0), &{\bf v}(t-k-3, 0), &{\bf v}(t-k-2, 0),\\
{\bf v}(t-k-1, 0), &{\bf v}(t-k, 1), &{\bf v}(t-k+1, 1)
\end{array}\right\}\subseteq \mathcal{D}^{-}(v), $$
thus
$$|S({\bf v}(2t-2k-3, 1))\cap 2^{\wedge}A|=2. $$
So $~{\bf v}(2t-2k-3, 1)\not\in W$. Since $t-k+1\geqslant 7$, we have $2t-2k-3>t-k+1$.
By (\ref{e4-6}) we have ${\bf v}(2t-2k-3, 1)\not\in A$, which contradicts with (\ref{e}).
Thus $$q_{1}=\cdots =q_{t-k+1}=0. $$
Hence
\begin{equation}\label{e4-01}\{{\bf v}(i, q_{i}): i\in [1, t-k+1]\}=\{i(w_{2}-w_{1}): i\in [1, t-k+1]\}. \end{equation}
It follows that
\begin{equation}\label{e4-9}a_{k-1}>(t-k+1)(w_{2}-w_{1}). \end{equation}
It is easy to see that
$${\bf v}(2k-4, q_{2k-4})=a_{k-1}-(w_{2}-w_{1}), $$
by (\ref{e4-8}) and (\ref{e4-9}) we have
\begin{equation}\label{e4-02}\left\{{\bf v}(i, q_{i}): i\in [t, 2k-4]\right\}=\left\{a_{k-1}-i(w_{2}-w_{1}): i\in [1, 2k-3-t]\right\}. \end{equation}

Noting that ${\bf v}(t, q_{t})=r_{v}(0)$, ${\bf v}(t-k+1, q_{t-k+1})=r_{v}(k-2)$, by (\ref{eq3.2-1})-(\ref{eq3.2-4}) and (\ref{e4-01}), (\ref{e4-02}) we have
$${\bf v}(t, q_{t})=a_{k-1}-(2k-3-t)(w_{2}-w_{1})=\frac{w_{2}}{2} \text{ or }\frac{w_{2}+a_{k-1}}{2}, $$
$${\bf v}(t-k+1, q_{t-k+1})=(t-k+1)(w_{2}-w_{1})=\frac{w_{1}}{2} \text{ or }\frac{w_{1}+a_{k-1}}{2}. $$
Subtracting the two equations gives $w_{2}-w_{1}=1, 2$ or $3$.
\begin{itemize}
  \item $w_{2}-w_{1}=1$. By (\ref{e4-01}) and (\ref{e4-02}), there exists an integer $k\leqslant \theta\leqslant 2k-4$ such that
$$A=[0, \theta-k+1]\cup [\theta, 2k-3]. $$
  \item $w_{2}-w_{1}=2$. By (\ref{e4-01}) and (\ref{e4-02}), we have
$$A=\{2i: i\in [0, t-k+1]\}\cup \{2k-3-2i: i\in [0, 2k-3-t]\}. $$
It follows that there exists an integer $1\leqslant \theta\leqslant k-3$ such that
$$A=\{2i, 2j-1: i\in [0, \theta], j\in[\theta+1, k-1]\}. $$
  \item $w_{2}-w_{1}=3$. Then $3\nmid 2k-3$. By (\ref{e4-01}) and (\ref{e4-02}), we have
$$A=\{3i: i\in [0, t-k+1]\}\cup \{2k-3-3i: i\in [0, 2k-3-t]\}. $$
It follows that there exists an integer $\frac{k-3}{3}<\theta<\frac{2k-3}{3}$ such that
$$A=\{3i, 3j-k: i\in [0, \theta], j\in[\theta+1, k-1]\}. $$
\end{itemize}

Similarly, if ${\bf v}(t-k, q_{t-k})$ and ${\bf v}(t-k+1, q_{t-k+1})$ satisfy $q_{t-k+1}=q_{t-k}+1$, then
$$\{{\bf v}(i, q_{i}): i\in [1, t-k+1]\}=\{i(w_{2}-w_{1})-(i-1)a_{k-1}: i\in [1, t-k+1]\}$$
and
$$\{{\bf v}(i, q_{i}): i\in [t, 2k-4]\}=\{ia_{k-1}-i(w_{2}-w_{1}): i\in[1, 2k-3-t]\}. $$
We can verify that $w_{2}-w_{1}=2k-5$, $2k-6$ or $2k-7$, there exists a new structure of $A$:
$$A=\left\{0, 4i, 4i-3, 2k-3: 1\leqslant i\leqslant \frac{k-2}{2}\right\}, \ \ 2\mid k. $$

{\bf Case 2. }$m=3$. Then $k\equiv 0\pmod{3}$ and $|U|=1$. Put $U=\{u\}$. By (\ref{e4-002}) we have
$$A=\{a_{k-1}\}\cup (u+H)\cup\mathcal{D}^{-}(v). $$
If $v\equiv 0\pmod{3}$, then $u=1$ or $2$. Without loss of generality, we consider $u=1$.
By (\ref{e4-002}) we have
$$1+3t\in A, \ \ 2+3t\not\in A, \ \ 0\leqslant t\leqslant \frac{a_{k-1}}{3}-1. $$
If $\frac{a_{k-1}}{3}-1\geqslant 3$, then
$$5=1+4\in 2^{\wedge}A, \ \ 5+a_{k-1}=7+(1+a_{k-1}-3)\in 2^{\wedge}A. $$
Thus $|S(5)\cap 2^{\wedge}A|=2$, which contradicts with (\ref{e4-2}).
If $\frac{a_{k-1}}{3}-1<3$, then $a_{k-1}=2k-3<12$, thus $k\leqslant 7$, which contradicts with $k\geqslant 15$.
Hence $v\not\equiv 0\pmod{3}$, and then $u=0$. Hence
$$A=\left\{3t: 0\leqslant t\leqslant \frac{2k-3}{3}\right\}\cup\mathcal{D}^{-}(v). $$

Since $k\geqslant 15$, we have $|\mathcal{D}^{-}(v)|\geqslant 5$. Write
$$\mathcal{D}^{-}(v)=\left\{r_{v}(x): x=0, \ldots, \frac{a_{k-1}}{6}-\frac{1}{2}\right\}. $$
If $\mathcal{D}^{-}(v)$ is not a monotone arithmetic progression,
then there exists an integer $j\in [0, \frac{a_{k-1}}{6}-\frac{7}{2}]$ such that one of the following cases holds:

(1) $r_{v}(j)<r_{v}(j+1), \ \ r_{v}(j+1)<r_{v}(j+2), \ \ r_{v}(j+2)>r_{v}(j+3)$;

(2) $r_{v}(j)<r_{v}(j+1), \ \ r_{v}(j+1)>r_{v}(j+2), \ \ r_{v}(j+2)>r_{v}(j+3)$;

(3) $r_{v}(j)>r_{v}(j+1), \ \ r_{v}(j+1)<r_{v}(j+2), \ \ r_{v}(j+2)<r_{v}(j+3)$;

(4) $r_{v}(j)>r_{v}(j+1), \ \ r_{v}(j+1)>r_{v}(j+2), \ \ r_{v}(j+2)<r_{v}(j+3)$;

(5) $r_{v}(j)<r_{v}(j+1), \ \ r_{v}(j+1)>r_{v}(j+2), \ \ r_{v}(j+2)<r_{v}(j+3)$;

(6) $r_{v}(j)>r_{v}(j+1), \ \ r_{v}(j+1)<r_{v}(j+2), \ \ r_{v}(j+2)>r_{v}(j+3)$.

The cases (1)-(4) are similar, we only consider (1). By the definition of $\mathcal{D}^{-}(v)$, we have
$$r_{v}(j+1)=r_{v}(j)+(w_{2}-w_{1}), $$
$$r_{v}(j+2)=r_{v}(j)+2(w_{2}-w_{1}), $$
$$r_{v}(j+3)=r_{v}(j)+3(w_{2}-w_{1})-a_{k-1}, $$
thus
$$r_{v}(j+1)+r_{v}(j+2)\in 2^{\wedge}A, $$
$$r_{v}(j+1)+r_{v}(j+2)-a_{k-1}=r_{v}(j)+r_{v}(j+3)\in 2^{\wedge}A. $$
Thus
$$|S(r_{v}(j+1)+r_{v}(j+2)-a_{k-1})\cap 2^{\wedge}A|=2. $$
Since $r_{v}(j+1)+r_{v}(j+2)-a_{k-1}\equiv 2v\pmod{3}$, we have $r_{v}(j+1)+r_{v}(j+2)-a_{k-1}\not\in A$, which contradicts with (\ref{e}).

The cases (5)-(6) are similar, we only consider (5). By the definition of $\mathcal{D}^{-}(v)$, we have
$$r_{v}(j+1)=r_{v}(j)+(w_{2}-w_{1}), $$
$$r_{v}(j+2)=r_{v}(j)+2(w_{2}-w_{1})-a_{k-1}, $$
$$r_{v}(j+3)=r_{v}(j)+3(w_{2}-w_{1})-a_{k-1}. $$
Since $|\mathcal{D}^{-}(v)|\geqslant 5$, we have
$$r_{v}(j+4)=r_{v}(j)+4(w_{2}-w_{1})-2a_{k-1}. $$
If not, then $r_{v}(j+4)=r_{v}(j)+4(w_{2}-w_{1})-a_{k-1}$, thus
$$r_{v}(j+1)>r_{v}(j+2), \ \ r_{v}(j+2)<r_{v}(j+3), \ \ r_{v}(j+3)<r_{v}(j+4), $$
which is similar to (3), a contradiction. Hence
$$r_{v}(j)+r_{v}(j+4)\in 2^{\wedge}A, $$
$$r_{v}(j)+r_{v}(j+4)+a_{k-1}=r_{v}(j+1)+r_{v}(j+3)\in 2^{\wedge}A. $$
Thus
$$|S(r_{v}(j)+r_{v}(j+4))\cap 2^{\wedge}A|=2. $$
Similar to (1), we have $r_{v}(j)+r_{v}(j+4)\not\in A$, which also contradicts with (\ref{e}).

Therefore, $\mathcal{D}^{-}(v)$ is a monotone arithmetic progression. To be specific, by the definition of $\mathcal{D}^{-}(v)$, we have
\begin{equation}\label{e4-4}\mathcal{D}^{-}(v)=\left\{r_{v}(i)=v+i(w_{2}-w_{1}): 0\leqslant i\leqslant \frac{a_{k-1}}{6}-\frac{1}{2}\right\}\end{equation}
or
\begin{equation}\label{e4-5}\mathcal{D}^{-}(v)=\left\{r_{v}(i)=v+i(w_{2}-w_{1})-ia_{k-1}: 0\leqslant i\leqslant \frac{a_{k-1}}{6}-\frac{1}{2}\right\}. \end{equation}

For case (\ref{e4-4}), $\mathcal{D}^{-}(v)$ is a monotone increasing arithmetic progression.
By the definition of $\mathcal{D}^{-}(v)$, we have $v=\frac{w_{2}}{2}$ and
$$\frac{w_{2}}{2}+\left(\frac{a_{k-1}}{6}-\frac{1}{2}\right)(w_{2}-w_{1})=\frac{w_{1}+a_{k-1}}{2}, $$
thus $w_{2}-w_{1}=3$. Therefore
$$A=\left\{3i: 0\leqslant i\leqslant \frac{2k-3}{3}\right\}\cup\left\{\theta+3i: 0\leqslant i\leqslant \frac{k-3}{3}\right\}, $$
where $2\leqslant \theta\leqslant k-2$ and $3\nmid\theta$.

For case (\ref{e4-5}), $\mathcal{D}^{-}(v)$ is a monotone decreasing arithmetic progression.
Similar to case (\ref{e4-4}), we have
$$A=\left\{3i: 0\leqslant i\leqslant \frac{2k-3}{3}\right\}\cup\left\{\theta+3i: 0\leqslant i\leqslant \frac{k-3}{3}\right\}, \ \ \theta=1, k-1.  $$

This completes the proof of Theorem \ref{T3}.

\section{Proof of Theorem \ref{T1}}\label{S7}
Let $k=3$. Then $A=\{0, 1, a_{2}\}$, thus
$$|2^{\wedge}A|=|\{1, a_{2}, a_{2}+1\}|=3>2, $$
so the result holds. Now we assume that $k\geqslant 4$.
By Theorem \ref{T2}, it is sufficiency to show that the Freiman-Lev conjecture is true for $a_{k-1}\geqslant 2k-2$, $a_{k-2}<2k-4$ and $a_{i}\geqslant 2i$ for some $i\in [1, k-3]$.

Choose $s\in [2, k-2]$ such that $a_{j}<2j$ for $j=s, \ldots, k-2$, and $a_{s-1}\geqslant 2(s-1)$. Then
$$2s-2\leqslant a_{s-1}<a_{s}<2s, $$
and so $a_{s-1}=2s-2$ and $a_{s}=2s-1$.

{\bf Case 1. }$s=k-2$. Let $A'=A\backslash\{a_{k-1}\}$. Then
$$k':=|A'|=k-1=s+1, $$
thus
$$a_{k'-2}=2k'-4, \ \ a_{k'-1}=2k'-3. $$
By Theorem C, we have
$$|2^{\wedge}A'|\geqslant 3k'-7=3k-10. $$
If $|2^{\wedge}A'|>3k-10$, then
$$|2^{\wedge}A|\geqslant |2^{\wedge}A'\cup \{a_{k-1}+a_{k-2}, a_{k-1}+a_{k-3}\}|\geqslant 3k-7. $$
Assume that $|2^{\wedge}A'|=3k-10$. Then
$$|2^{\wedge}A|\geqslant |2^{\wedge}A'\cup \{a_{k-1}+a_{k-2}, a_{k-1}+a_{k-3}\}|\geqslant 3k-8. $$
If
$$a_{k-4}=a_{k'-3}\geqslant 2k'-6=2k-8, $$
then
$$a_{k-1}+a_{k-4}\geqslant 4k-10. $$
Since the largest integer of $2^{\wedge}A'$ is
$$a_{k'-2}+a_{k'-1}=2k'-4+2k'-3=4k-11, $$
we have
$$a_{k-1}+a_{k-4}\not\in 2^{\wedge}A', \ \ a_{k-1}+a_{k-4}\in 2^{\wedge}A. $$
Hence $|2^{\wedge}A|\geqslant 3k-7$. If
$$a_{k-4}=a_{k'-3}<2k'-6=2k-8, $$
combined with $a_{k'-2}=2k'-4$ and $a_{k'-1}=2k'-3$, then the structure of set $A'$ can only be as described in Remark \ref{r1}.
In these situations, with some appropriate calculations, it is easy to verify that $|2^{\wedge}A|\geqslant 3k-7$ when $a_{k-1}\geqslant 2k-2$.

{\bf Case 2. }$s<k-2$. Since $a_{s+1}<2s+2$, we have $a_{s+1}=2s$ or $2s+1$.
Define the sets $A_{1}$ and $A_{2}$ by
$$A_{1}=\{a_{0}, a_{1}, \ldots, a_{s-1}, a_{s}, a_{s+1}\}, $$
$$A_{2}=\{a_{s-1}, a_{s}, a_{s+1}, \ldots, a_{k-2}, a_{k-1}\}. $$
Since $a_{s}-a_{s-1}=1$, we have $\gcd(A_{1})=\gcd(A_{2})=1$.

Let $k_{1}:=|A_{1}|=s+2$. Then $4\leqslant k_{1}\leqslant k-1$ and
$$a_{k_{1}-1}=a_{s+1}=2k_{1}-4 \text{ or } 2k_{1}-3. $$
By Theorem B, we have
\begin{equation}\label{e2-1}|2^{\wedge}A_{1}|\geqslant 3k_{1}-7=3s-1. \end{equation}

Define the set $A^{*}_{2}$ by
$$A^{*}_{2}=A_{2}-a_{s-1}=\{0, 1, \ldots, a_{k-1}-a_{s-1}\}. $$
Let $k_{2}=|A^{*}_{2}|=|A_{2}|=k-s+1$. Then $4\leqslant k_{2}\leqslant k-1$.
Since $a_{s-1}=2s-2$, we have
$$a_{k-1}-a_{s-1}\geqslant 2k-2-(2s-2)=2(k-s)=2k_{2}-2. $$
For any $j=1, \ldots, k_{2}-2$, we have
$$a_{s+j-1}-a_{s-1}<2(s+j-1)-(2s-2)=2j. $$
By Theorem \ref{T2}, we have
\begin{equation}\label{e2-2}|2^{\wedge}A_{2}|=|2^{\wedge}A^{*}_{2}|\geqslant 3k_{2}-7=3k-3s-4. \end{equation}

Since
$$2^{\wedge}A_{1}\cup2^{\wedge}A_{2}\subseteq 2^{\wedge}A$$
and
$$2^{\wedge}A_{1}\cap2^{\wedge}A_{2}=\{a_{s-1}+a_{s}, a_{s-1}+a_{s+1}, a_{s}+a_{s+1}\}, $$
by (\ref{e2-1}) and (\ref{e2-2}) we have
$$|2^{\wedge}A|\geqslant 3k-8. $$

Obviously, if $|2^{\wedge}A_{1}|>3s-1$ or $|2^{\wedge}A_{2}|=|2^{\wedge}A^{*}_{2}|>3k-3s-4$, then $|2^{\wedge}A|\geqslant 3k-7$. Now, we assume that
$$|2^{\wedge}A_{1}|=3s-1, \ \ |2^{\wedge}A_{2}|=|2^{\wedge}A^{*}_{2}|=3k-3s-4. $$
By Theorem \ref{T2}, we have $k_{2}\geqslant 6$ and
$$A_{2}^{*}=\left\{3i, 3i+1, k_{2}+1+3i: i=0, \ldots, \frac{k_{2}-3}{3}\right\}, \ \ k_{2}\equiv 0\pmod{3}$$
or
$$A_{2}^{*}=\left\{3i, 3i+1, k_{2}-1+3i, 2k_{2}-2: i=0, \ldots, \frac{k_{2}-4}{3}\right\}, \ \ k_{2}\equiv 1\pmod{3}, $$
thus $0, 1, 3, 4\in A_{2}^{*}$, it implies that
$$a_{k_{1}-3}=a_{s-1}=2s-2=2k_{1}-6, $$
$$a_{k_{1}-2}=a_{s}=2s-1=2k_{1}-5, $$
$$a_{k_{1}-1}=a_{s+1}=2s+1=2k_{1}-3$$
and $a_{s+2}=2s+2$. By Remark \ref{r2}, we have
$$a_{s-2}=a_{k_{1}-4}=2k_{1}-8=2s-4, $$
thus
$$2a_{s-1}+2=4s-2=a_{s-2}+a_{s+2}\in 2^{\wedge}A. $$
By Remark \ref{r2.1}, we have $2\not\in 2^{\wedge}A^{*}_{2}$, thus $2a_{s-1}+2\not\in 2^{\wedge}A_{2}$. Clearly, $2a_{s-1}+2\not\in 2^{\wedge}A_{1}$.
Hence
$$|2^{\wedge}A|\geqslant 3k-7. $$

This completes the proof of Theorem \ref{T1}.

\section{Appendix: Proofs of Lemmas \ref{P1}-\ref{P3} and Proposition \ref{P4}}\label{A}
\begin{proof}[Proof of Lemma \ref{P1}]
Let $c_{1}, c_{2}\in [2k-3, 2k-4+b_{m-1}]$ such that $c_{2}=2k-4+c$, $c_{1}=2k-5+c$, where $c\in [2, b_{m-1}]$.
By Lemma \ref{L2-2} (iii), we know that for any $b\in B\backslash\{b_{m}\}$, we have
\begin{equation}\label{e2.6}\left|[b+1, 2k-5+c-b]\cap A'\right|=k-2+\left\lfloor\frac{c}{2}\right\rfloor-b, \end{equation}
\begin{equation}\label{e2.7}|\{i, 2k-4+c-i\}\cap A'|=1, \ \ i=b+1, \ldots, k-2+\left\lfloor \frac{c}{2}\right\rfloor, \end{equation}
\begin{equation}\label{e2.8}|[b+1, 2k-6+c-b]\cap A'|=k-2+\left\lfloor\frac{c-1}{2}\right\rfloor-b, \end{equation}
\begin{equation}\label{e2.9}|\{i, 2k-5+c-i\}\cap A'|=1, \ \ i=b+1, \ldots, k-2+\left\lfloor \frac{c-1}{2}\right\rfloor. \end{equation}

If $c$ is even, then by (\ref{e2.6}) and (\ref{e2.8}), choose $b=b_{m-1}$, we have
$$2k-5+c-b_{m-1}\in A'. $$
By (\ref{e2.7}) and $b_{m-1}+1\in A'$, we have
$$2k-5+c-b_{m-1}=b_{m-1}+1, $$
thus
$$b_{m-1}=k-3+\frac{c}{2}\geqslant k-2. $$
By Lemma \ref{L2-3}, we have $b_{m-1}<\frac{b_{m}}{2}\leqslant k-2$, a contradiction.

Let $2\nmid c$. Then $c\geqslant 3$. If $c\in (1, b_{m-2}]$, then by (\ref{e2.7}) and (\ref{e2.9}), choose $b=b_{m-2}$, we have
$$b_{m-2}+1\in A' \overset{(\ref{e2.9})}\Longleftrightarrow 2k-6+c-b_{m-2}\not\in A' \overset{(\ref{e2.7})}\Longleftrightarrow b_{m-2}+2\in A'. $$
Repeating the above process, we have
$$b_{m-2}+1\in A'\Longleftrightarrow b_{m-2}+2\in A'\Longleftrightarrow \cdots \Longleftrightarrow k-2+\frac{c-1}{2}\in A'. $$
That is,
$$\left[b_{m-2}+1, k-2+\frac{c-1}{2}\right]\subseteq A', $$
thus
$$\left[2b_{m-2}+3, 2k-6+c\right]\subseteq 2^{\wedge}A'. $$
By Lemma \ref{L2-3} and $b_{m}\leqslant 2k-4$, we have
$$2b_{m-2}+2\leqslant b_{m-1}<b_{m}<2k-6+c, $$
thus $b_{m}\in 2^{\wedge}A'$, a contradiction.

If $c\in (b_{m-2}, b_{m-1}]$ is an odd integer, then similar to the above discuss, we have
\begin{equation}\label{e2.10}\left[b_{m-1}+1, k-2+\frac{c-1}{2}\right]\subseteq A', \end{equation}
thus
\begin{equation}\label{e2.11}\left[2b_{m-1}+3, 2k-6+c\right]\subseteq 2^{\wedge}A'. \end{equation}
Noting that $b_{m}\not\in A'$, $b_{m}\leqslant 2k-4<2k-6+c$ and $b_{m}\geqslant 2b_{m-1}+2$, by (\ref{e2.11}) we have
$$b_{m}=2b_{m-1}+2. $$

By (\ref{e2.7}) and (\ref{e2.10}), we have
\begin{equation}\label{e2.12}\left[k-1+\frac{c-1}{2}, 2k-5+c-b_{m-1}\right]\cap A'=\emptyset. \end{equation}
By (\ref{e2.10}), (\ref{e2.12}) and Lemma \ref{L2-2} (ii), we have
\begin{eqnarray*}
& &\left|\left[b_{m}-(2k-5+c-b_{m-1}), b_{m-1}\right]\cap A'\right|\\
&=&
\left|\left[b_{m}-(2k-5+c-b_{m-1}), b_{m}-\left(k-1+\frac{c-1}{2}\right)\right]\right|\\
&=&
k-3-b_{m-1}+\frac{c+1}{2}.
\end{eqnarray*}
If $b_{m}=2k-4$, then
$$b_{m}>2k-5+c-b_{m-1}. $$
If $b_{m}<2k-4$, then by Lemma \ref{L2-2} (i) and (iv), we have $b_{m}\not\in A'$ and $b_{m}+1\in A'$.
By (\ref{e2.10}) and (\ref{e2.12}), we have
$$b_{m}\geqslant 2k-5+c-b_{m-1}. $$

If $b_{m}>2k-5+c-b_{m-1}$, then again by Lemma \ref{L2-2} (ii), we have
\begin{eqnarray*}
&&\left|\left[1, b_{m}-(2k-4+c-b_{m-1})\right]\cap A'\right|\\
&=&\frac{b_{m-1}}{2}-\left(k-3-b_{m-1}+\frac{c+1}{2}\right)\\
&=&
\frac{1}{2}\left(b_{m}-(2k-3+c-b_{m-1})\right).
\end{eqnarray*}
Hence,
$$a_{\frac{1}{2}\left(b_{m}-(2k-3+c-b_{m-1})\right)+1}=b_{m}-(2k-5+c-b_{m-1}), $$
which contradicts with $a_{i}<2i$ for all $i=1, \ldots, k-2$.

If $b_{m}=2k-5+c-b_{m-1}$, then
$$[0, b_{m}]\cap A'=\left[0,\frac{b_{m-1}}{2}\right]\cup \left[b_{m-1}+1, \frac{3b_{m-1}}{2}+1\right], $$
$$b_{m}+1=2b_{m-1}+3\in A', $$
it implies that $m=2$. Moreover,
$$[1, b_{1}-1]\cup [b_{1}+1, 2b_{1}+1]\cup [2b_{1}+3, 3b_{1}+1]\subseteq 2^{\wedge}A'. $$

Noting that
$$2k-5+c=3b_{1}+2, \ \ 2k-4+c=3b_{1}+3. $$
Since $c\geqslant 3$, we have $3b_{1}+1\geqslant 2k-3$ and $3b_{1}+1\in 2^{\wedge}A'$, it follows that
$$[1, 2k-4]\backslash \{b_{1}, b_{2}\}\subseteq 2^{\wedge}A', \ \ 2k-3\in 2^{\wedge}A', \ \ a_{k-1}+A'\subseteq 2^{\wedge}A, $$
thus $|2^{\wedge}A|\geqslant 3k-6$.

This completes the proof of Lemma \ref{P1}.
\end{proof}

\begin{proof}[Proof of Lemma \ref{P2}]
Let $c_{1}, c_{2}\in [2k-3, 2k-4+b_{m-1}]$ such that $c_{2}=2k-4+c$, $c_{1}=2k-6+c$, where $c\in [3, b_{m-1}]$.

If $c$ is even, then $c\geqslant 4$. By Lemma \ref{L2-2} (iii) we have
$$k-3+\frac{c}{2}\in A', \ \ k-2+\frac{c}{2}\in A'. $$
Similar to Lemma ~\ref{P1}, we have
$$\left\{d\in \left[b_{m-1}+1, k-3+\frac{c}{2}\right]: 2\nmid d\right\}\subseteq A'. $$
If $k-4+\frac{c}{2}\in A'$, by $~k-2+\frac{c}{2}\in A'$ we have $2k-6+c\in 2^{\wedge}A'$, a contradiction.
Hence, $k-4+\frac{c}{2}\not\in A'$. Again by \ref{L2-2} (iii) we have
$$k-4+\frac{c}{2}\not\in A' \overset{c_{2}\not\in 2^{\wedge}A'}\Longleftrightarrow k+\frac{c}{2}\in A'\Longleftrightarrow \cdots \Longleftrightarrow b_{m-1}+2\not\in A', $$
that is
\begin{eqnarray}\label{e2.15}\left\{d\in \left[b_{m-1}+1, k-3+\frac{c}{2}\right]: 2\mid d\right\}\cap A'=\emptyset. \end{eqnarray}
By Lemma \ref{L2-2} (i) and \ref{L2-3} we have $\frac{b_{m}}{2}\in A'$, $\frac{b_{m}}{2}\geqslant b_{m-1}+1$, thus $2\nmid \frac{b_{m}}{2}$.
By $b_{m}\leqslant 2k-4$ and $c\geqslant 4$ we have
\begin{eqnarray}\label{e2.16}b_{m-1}+1\leqslant\frac{b_{m}}{2}\leqslant k-2<k-3+\frac{c}{2}. \end{eqnarray}
If $\frac{b_{m}}{2}=b_{m-1}+1$, then $b_{m}=2b_{m-1}+2$, thus
$$b_{m}-b_{m-1}=b_{m-1}+2\in A', $$
which contradicts with (\ref{e2.15}). Hence, $\frac{b_{m}}{2}>b_{m-1}+1$.
By (\ref{e2.16}) and $2\nmid \frac{b_{m}}{2}$, $2\nmid b_{m-1}+1$, $2\nmid k-3+\frac{c}{2}$ we have
$$\frac{b_{m}}{2}-2, \ \ \frac{b_{m}}{2}+2\in A', $$
thus $~b_{m}\in 2^{\wedge}A'$, a contradiction.

If $c$ is odd. Similar to the discussion of Lemma \ref{P1}, we have
$$\left\{d\in \left[b_{m-1}+1, k-2+\frac{c-1}{2}\right]: 2\nmid d\right\}\subseteq A', $$
$$\left\{d\in \left[b_{m-1}+1, k-2+\frac{c-1}{2}\right]: 2\mid d\right\}\cap A'=\emptyset. $$
The remainder proof is similar to Case 1, except for $c=3$ and $b_{m}=2k-4$,
which implies that $k-2\in A'$, thus $2\nmid k-2$, and then $k-3$, $k-1\not\in A'$, which contradicts with Lemma \ref{L2-2} (ii).

This completes the proof of Lemma \ref{P2}.
\end{proof}

\begin{proof}[Proof of Lemma \ref{P3}]
Let $c_{1}, c_{2}\in [2k-3, 2k-4+b_{m-1}]$ such that $c_{2}=2k-4+c$, $c_{1}=2k-7+c$, where $c\in [4, b_{m-1}]$.
Choose $b=b_{m-1}$, by Lemma \ref{L2-2} (iii), we have
\begin{equation}\label{e2.17}|[b_{m-1}+1, 2k-5+c-b_{m-1}]\cap A'|=k-2+\left\lfloor\frac{c}{2}\right\rfloor-b_{m-1}, \end{equation}
\begin{equation}\label{e2.17-1}|[b_{m-1}+1, 2k-8+c-b_{m-1}]\cap A'|=k-2+\left\lfloor\frac{c-3}{2}\right\rfloor-b_{m-1}. \end{equation}
Moreover, by $b_{m-1}+1\in A'$, we know that
$$2k-5+c-b_{m-1}\not\in A'. $$

Now, we shall show that $c$ is odd. If $c$ is even, then by (\ref{e2.17})-(\ref{e2.17-1}) we have
$$2k-7+c-b_{m-1}, \ \ 2k-6+c-b_{m-1}\in A'. $$
Thus $b_{m-1}+2\not\in A'$. Since $b_{m-1}<k-2$ and $c\geqslant 4$, we have
$$b_{m-1}+2<2k-6+c-b_{m-1}. $$
According to the parity of $b_{m-1}+2$ and $2k-7+c-b_{m-1}$ we have
$$b_{m-1}+2\neq 2k-7+c-b_{m-1}. $$

If $b_{m-1}+3<2k-7+c-b_{m-1}$, then again by Lemma \ref{L2-2} (iii), we have $b_{m-1}+3\not\in A'$.
By Lemma \ref{L2-2} (iv), $a_{\frac{b_{m-1}}{2}+1}=b_{m-1}+1$, thus
$$a_{\frac{b_{m-1}}{2}+2}\geqslant b_{m-1}+4, $$
a contradiction.

If $b_{m-1}+3=2k-7+c-b_{m-1}$, then $c=4$ and $b_{m-1}=k-3\not\in A'$, thus $b_{m}=2k-4$, it implies that $k-2\in A'$ and $k-1\in A'$, which contradicts with $2k-7+c=2k-3\not\in 2^{\wedge}A'$.

Hence, $c$ is odd and
$$|\{2k-7+c-b_{m-1}, 2k-6+c-b_{m-1}\}\cap A'|=1. $$
The cases (i)-(vi) are similar, for convenience, we only proof (i).

Assume that
$$2k-7+c-b_{m-1}\in A', \ \ 2k-6+c-b_{m-1}\not\in A'. $$
By Lemma \ref{L2-2} (iii), we have $b_{m-1}+3\not\in A'$. Since $a_{\frac{b_{m-1}}{2}+2}<b_{m-1}+4$,
we have
$$a_{\frac{b_{m-1}}{2}+2}=b_{m-1}+2\in A'. $$
Thus
$$b_{m-1}+1\in A', \ \ b_{m-1}+2\in A', \ \ b_{m-1}+3\not\in A'. $$
Similar to the discussion of Proposition \ref{P1}, we have
\begin{equation}\label{e2.18}\left\{b_{m-1}+1+3t: t\in \mathbb{N}\right\}\cap \left[b_{m-1}+1, k-2+\frac{c-1}{2}\right]\subseteq A', \end{equation}
\begin{equation}\label{e2.19}\left\{b_{m-1}+2+3t: t\in \mathbb{N}\right\}\cap \left[b_{m-1}+1, k-2+\frac{c-1}{2}\right]\subseteq A', \end{equation}
\begin{equation}\label{e2.20}\left(\left\{b_{m-1}+3t: t\in \mathbb{N}\right\}\cap \left[b_{m-1}+1, k-2+\frac{c-1}{2}\right]\right)\cap A'=\emptyset. \end{equation}

Since $c$ is odd and $2k-7+c\not\in 2^{\wedge}A'$, by Lemma \ref{L2-2} (iii), we have $k-3+\frac{c-1}{2}\in A'$.
By (\ref{e2.18}) and (\ref{e2.19}) we have
$$k-3+\frac{c-1}{2}\equiv b_{m-1}+1 \text{ or } b_{m-1}+2\pmod{3}. $$

If $k-3+\frac{c-1}{2}\equiv b_{m-1}+1\pmod{3}$, then $k-2+\frac{c-1}{2}\equiv b_{m-1}+2\pmod{3}$.
By (\ref{e2.19}) we have
$$k-2+\frac{c-1}{2}\in A'. $$
Since $k-4+\frac{c-1}{2}\equiv b_{m-1}\pmod{3}$, by (\ref{e2.20}) we have $k-4+\frac{c-1}{2}\not\in A'$, it implies that
$$k+1+\frac{c-1}{2}\in A'. $$
Repeating the above process, we have
\begin{equation}\label{e2.21}\left\{b_{m-1}+2+3t: t\in \mathbb{N}\right\}\cap \left[k+\frac{c-1}{2}, 2k-7+c-b_{m-1}\right]\subseteq A'. \end{equation}
By (\ref{e2.19}) and (\ref{e2.21}), we have
$$\left\{b_{m-1}+2+3t: t\in \mathbb{N}\right\}\cap \left[b_{m-1}+1, 2k-7+c-b_{m-1}\right]\subseteq A'. $$
By Lemma \ref{L2-2} (ii) and (\ref{e2.17}), we have
\begin{eqnarray*}
&&\left|[1, 2k-5+c-b_{m-1}]\cap A'\right|\\
&=&\left|[1, b_{m-1}]\cap A'\right|+\left|[b_{m-1}+1, 2k-5+c-b_{m-1}]\cap A'\right|\\
&=&k-2+\frac{c-1}{2}-\frac{b_{m-1}}{2} \end{eqnarray*}
and
$$a_{k-2+\frac{c-1}{2}-\frac{b_{m-1}}{2}+1}<2\left(k-2+\frac{c-1}{2}-\frac{b_{m-1}}{2}+1\right)=2k-3+c-b_{m-1}, $$
thus
$$a_{k-2+\frac{c-1}{2}-\frac{b_{m-1}}{2}+1}=2k-4+c-b_{m-1}\in A'. $$
Hence
\begin{eqnarray}\label{e2.22} \nonumber
&&\left[b_{m-1}+1, 2k-4+c-b_{m-1}\right]\cap A'\\
&=&\left(\left\{b_{m-1}+1+3t: t\in \mathbb{N}\right\}\cap \left[b_{m-1}+1, k-2+\frac{c-1}{2}\right]\right)\\ \nonumber
&&\cup\left(\left\{b_{m-1}+2+3t: t\in \mathbb{N}\right\}\cap \left[b_{m-1}+1, 2k-4+c-b_{m-1}\right]\right).  \nonumber
\end{eqnarray}

Similarly, if $k-3+\frac{c-1}{2}\equiv b_{m-1}+2\pmod{3}$, then
\begin{eqnarray}\label{e2.23} \nonumber
&&\left[b_{m-1}+1, 2k-4+c-b_{m-1}\right]\cap A'\\
&=&\left(\left\{b_{m-1}+2+3t: t\in \mathbb{N}\right\}\cap \left[b_{m-1}+1, k-2+\frac{c-1}{2}\right]\right)\\ \nonumber
&&\cup\left(\left\{b_{m-1}+1+3t: t\in \mathbb{N}\right\}\cap \left[b_{m-1}+1, 2k-4+c-b_{m-1}\right]\right).  \nonumber
\end{eqnarray}
Since $b_{m}\geqslant 2b_{m-1}+2$, we divide into the following cases:

{\bf Case 1. }$\frac{b_{m}}{2}\geqslant b_{m-1}+3$. Since $c\in[4, b_{m-1}]$ is odd and $b_{m}\leqslant 2k-4$, we have
$$\frac{b_{m}}{2}<k-3+\frac{c-1}{2}. $$

If $\frac{b_{m}}{2}\equiv k-3+\frac{c-1}{2}\pmod{3}$, then $\frac{b_{m}}{2}+3\leqslant k-3+\frac{c-1}{2}$.
By (\ref{e2.18}) and (\ref{e2.19}), we have $\frac{b_{m}}{2}-3\in A'$ and $\frac{b_{m}}{2}+3\in A'$,
thus $b_{m}\in 2^{\wedge}A'$, a contradiction.

Assume that $\frac{b_{m}}{2}\not\equiv k-3+\frac{c-1}{2}\pmod{3}$.
If $k-3+\frac{c-1}{2}\equiv b_{m-1}+1\pmod{3}$, then $\frac{b_{m}}{2}\equiv b_{m-1}+2\pmod{3}$.
By (\ref{e2.22}) we have $\frac{b_{m}}{2}-3\in A'$ and $\frac{b_{m}}{2}+3\in A'$, thus $b_{m}\in 2^{\wedge}A'$, a contradiction.
If $k-3+\frac{c-1}{2}\equiv b_{m-1}+2\pmod{3}$, then the proof is similar. Here, we omitted it.

{\bf Case 2. }$\frac{b_{m}}{2}=b_{m-1}+1$. Since $b_{m-1}+3\not\in A'$, by Lemma \ref{L2-2}, we have
$$b_{m-1}-1=b_{m}-(b_{m-1}+3)\in A'. $$
Since $a_{1}<2$, we have $a_{1}=1$, thus $b_{m-1}=1+(b_{m-1}-1)\in 2^{\wedge}A'$, a contradiction.

{\bf Case 3. }$\frac{b_{m}}{2}=b_{m-1}+2$.

For convenience, we assume that $k-3+\frac{c-1}{2}\equiv b_{m-1}+2\pmod{3}$. Then there exists a positive integer $t$ such that
$$k-3+\frac{c-1}{2}=b_{m-1}+2+3t, $$
thus
$$2k-4+c-b_{m-1}=b_{m-1}+1+6(t+1). $$
By (\ref{e2.19}) and (\ref{e2.23}), we have
$$A'_{1}:=\left\{b_{m-1}+2, b_{m-1}+5, \ldots, k-3+\frac{c-1}{2}\right\}\subseteq A', |A_{1}'|=t+1, $$
$$A'_{2}:=\left\{b_{m-1}+1, b_{m-1}+4, \ldots, 2k-4+c-b_{m-1}\right\}\subseteq A', |A_{2}'|=2(t+1)+1. $$

If $b_{m}\geqslant 2k-3+c-b_{m-1}$, then by $b_{m}=2b_{m-1}+4$ and
$$2k-5+c-b_{m-1}=b_{m-1}+6(t+1)\not\in A', $$
we have
$$b_{m-1}+4-6(t+1)=b_{m}-(b_{m-1}+6(t+1))\in A'. $$
By Lemma \ref{L2-2} (ii), we have
\begin{eqnarray*}
|[b_{m-1}+4-6(t+1), b_{m-1}]\cap A'|=3(t+1)-1,
\end{eqnarray*}
thus
$$|[1, b_{m-1}+3-6(t+1)]|=\frac{b_{m-1}}{2}-3(t+1)+1. $$
Hence
$$a_{\frac{b_{m-1}}{2}-3(t+1)+2}=b_{m-1}+4-6(t+1), $$
which contradicts with $a_{i}<2i$ for all $i=1, \ldots, k-2$.

If $b_{m}\leqslant 2k-4+c-b_{m-1}$, by (\ref{e2.19}), (\ref{e2.23}) and $b_{m}\not\in A'$, we have
$$b_{m}\equiv b_{m-1} \text{ or } b_{m-1}+2\pmod{3}. $$
If $b_{m}\equiv b_{m-1}+2\pmod{3}$, then $b_{m}+1\equiv b_{m-1}+3\pmod{3}$, thus $b_{m-1}+1\not\in A'$, a contradiction.
Hence, $b_{m}\equiv b_{m-1}\pmod{3}$. Then there exists a positive integer $s$ such that
$$b_{m}=b_{m-1}+3s, \ \ 0<s\leqslant 2t+2. $$
If $0<s\leqslant t+1$, then
\begin{eqnarray*}
&&|[b_{m-1}+1, b_{m}]\cap A'|\\
&=&|[b_{m-1}+1, b_{m}]\cap A'_{1}|+|[b_{m-1}+1, b_{m}]\cap A'_{2}|\\
&=&2s>\frac{b_{m}-b_{m-1}}{2},
\end{eqnarray*}
thus
$$|[1, b_{m}]\cap A'|=|[1, b_{m-1}]\cap A'|+|[b_{m-1}+1, b_{m}]\cap A'|>\frac{b_{m}}{2}, $$
which contradicts with Lemma \ref{L2-2} (ii).

If $t+1<s<2(t+1)$, then
\begin{eqnarray*}
&&|[b_{m-1}+1, b_{m}]\cap A'|\\
&=&|[b_{m-1}+1, b_{m}]\cap A'_{1}|+|[b_{m-1}+1, b_{m}]\cap A'_{2}|\\
&=&t+s+1>\frac{b_{m}-b_{m-1}}{2},
\end{eqnarray*}
thus
$$|[1, b_{m}]\cap A'|=|[1, b_{m-1}]\cap A'|+|[b_{m-1}+1, b_{m}]\cap A'|>\frac{b_{m}}{2}, $$
which also contradicts with Lemma \ref{L2-2} (ii).

Hence, $s=2(t+1)$. That is,
$$b_{m}=2k-5+c-b_{m-1}=b_{m-1}+6(t+1)<2k-4. $$
It is easy to see that $b_{m-1}\geqslant c\geqslant 5$.

By $b_{m}=2b_{m-1}+4$ and Lemma \ref{L2-2} (ii), we have $b_{m-1}\equiv 2\pmod{3}$ and
\begin{eqnarray*}
[0, b_{m}]\cap A'
&=&\left\{3i: 0\leqslant i\leqslant \frac{2b_{m-1}+2}{3}\right\}\\
&&\cup \left\{3i+1: 0\leqslant i\leqslant \frac{b_{m-1}-2}{6}, \frac{b_{m-1}+1}{3}\leqslant i\leqslant \frac{b_{m-1}}{2}\right\},
\end{eqnarray*}
thus $m=3$ and $b_{1}=2$.

Next, we show that $|2^{\wedge}A|\geqslant 3k-6$.

It is easy to see that
\begin{equation}\label{e2.24}\left\{3i: 1\leqslant i\leqslant \frac{4b_{2}+1}{3}\right\}\subseteq 2^{\wedge}A', \end{equation}
\begin{equation}\label{e2.25}\left\{3i+1: 0\leqslant i\leqslant \frac{7b_{2}+4}{6}\right\}\subseteq 2^{\wedge}A', \end{equation}
\begin{equation}\label{e2.26}\left\{3i+2: 1\leqslant i\leqslant b_{2}-1, i\neq \frac{b_{2}-2}{3}, \frac{2b_{2}+2}{3}\right\}\subseteq 2^{\wedge}A'. \end{equation}

Noting that
$$2k-7+c=3b_{2}+2, \ \ 2k-4+c=3b_{2}+5$$
and
\begin{equation}\label{e2.27}[1, 2k-4]\backslash \{2, b_{2}, b_{3}\}\subseteq 2^{\wedge}A', \ \ a_{k-1}+A'\subseteq 2^{\wedge}A. \end{equation}

If $c\geqslant 9$, then by (\ref{e2.24})-(\ref{e2.26}), we have
$$2k-3, 2k-2, 2k-1, 2k, 2k+1\in 2^{\wedge}A'. $$
thus there exist two different elements of $2^{\wedge}A$ from (\ref{e2.27}),
namely that $\{2k-3, 2k-2\}\subseteq 2^{\wedge}A'$ for $a_{k-1}>2k-2$ and $\{2k-3, 2k\}\subseteq 2^{\wedge}A'$ for $a_{k-1}=2k-2$, respectively.
Hence, $|2^{\wedge}A|\geqslant 3k-6$.

If $c=7$, then $3b_{2}+2=2k$ and
$$2k-3, 2k-2, 2k-1\in 2^{\wedge}A'. $$
Similar to the above, we have $|2^{\wedge}A|\geqslant 3k-6$ for $a_{k-1}>2k-2$.
Assume that $a_{k-1}=2k-2$. Then $a_{k-1}=3b_{2}$. Put
$$c':=3\left(\frac{7b_{2}+4}{6}+1\right)+1=3b_{2}+\frac{b_{2}}{2}+6, $$
$$c'':=3\left(\frac{7b_{2}+4}{6}\right)+2=3b_{2}+\frac{b_{2}}{2}+4. $$

If $b_{2}\geqslant 16$, then $\frac{b_{2}}{2}+6\not\in A'$ and $\frac{b_{2}}{2}+4\not\in A'$.
If $c'\in 2^{\wedge}A'$ or $c''\in 2^{\wedge}A'$, then
$$[1, 2k-4]\backslash \{2, b_{2}, b_{3}\}\subseteq 2^{\wedge}A', \ \ 2k-3\in 2^{\wedge}A', $$
$$\{c', c''\}\cap 2^{\wedge}A'\neq \emptyset, \ \ a_{k-1}+A'\subseteq 2^{\wedge}A, $$
thus $|2^{\wedge}A|\geqslant 3k-6$.
Assume that $c', c''\not\in 2^{\wedge}A'$. Since $3b_{2}+2=2k$, we have
$$3b_{2}+\frac{b_{2}}{2}+6=2k-4+\frac{b_{2}}{2}+8\leqslant 2k-4+b_{2}, $$
thus $c', c''\in (2k-4, 2k-4+b_{2}]$, but $c'-c''=2$, which contradicts with Lemma \ref{P2}.
Hence, $b_{2}<16$. Since $b_{m-1}\equiv 2\pmod{3}$ and $b_{2}$ is even, we have $b_{2}=8$ or $14$.
It is easy to verify that these special cases are true.

If $c=5$, then the proof is similar to $c=7$. Here, we omit it.

This completes the proof of Lemma \ref{P3}.
\end{proof}

\begin{proof}[Proof of Proposition \ref{P4}]
(Sufficiency). It is obvious.

(Necessity). By Lemma \ref{L2-2} (ii), we have
$$|[0, b_{m-1}]\cap A'|=\frac{b_{m-1}}{2}+1, \ \ a_{\frac{b_{m-1}}{2}+1}=b_{m-1}+1\in A', $$
thus
\begin{equation}\label{e2.28}\left|[b_{m-1}+1, 2k-4]\cap A'\right|=k-2-\frac{b_{m-1}}{2}. \end{equation}
Since $2k-3+b_{m-1}\not\in 2^{\wedge}A'$, we have
$$|\{i, 2k-3+b_{m-1}-i\}\cap A'|\leqslant 1, \ \ i=b_{m-1}+1, \ldots, k-2+\frac{b_{m-1}}{2}, $$
thus
\begin{equation}\label{e2.29}|[b_{m-1}+1, 2k-4]\cap A'|\leqslant k-2-\frac{b_{m-1}}{2}. \end{equation}
By (\ref{e2.28}) and (\ref{e2.29}) we have
\begin{eqnarray}\label{e2.30}|\{i, 2k-3+b_{m-1}-i\}\cap A'|=1, \ \ i=b_{m-1}+1, \ldots, k-2+\frac{b_{m-1}}{2}. \end{eqnarray}
Since $b_{m-1}+1\in A'$, we have
$$|[b_{m-1}+2, 2k-4]\cap A'|=k-3-\frac{b_{m-1}}{2}. $$
Since $2k-2+b_{m-1}\not\in 2^{\wedge}A'$, we have
$$|\{i, 2k-2+b_{m-1}-i\}\cap A'|\leqslant 1, \ \ i=b_{m-1}+2, \ldots, k-1+\frac{b_{m-1}}{2}, $$
thus
$$|[b_{m-1}+2, 2k-4]\cap A'|\leqslant k-2-\frac{b_{m-1}}{2}. $$
Hence, there exists a unique integer $t\in\{2, \ldots, k-1-\frac{b_{m-1}}{2}\}$ such that
$$b_{m-1}+t\not\in A', \ \ 2k-2-t\not\in A'. $$
Moreover, for any $i\in \left[b_{m-1}+2, k-1+\frac{b_{m-1}}{2}\right]\backslash\{b_{m-1}+t\}$, we have
\begin{eqnarray}\label{e2.31}\{i, 2k-2+b_{m-1}-i\}\cap A'=1. \end{eqnarray}
By (\ref{e2.30}) and (\ref{e2.31}), we have
\begin{eqnarray*}
b_{m-1}+1\in A' &\overset{(\ref{e2.30})}\Longleftrightarrow& 2k-4\not\in A' \overset{(\ref{e2.31})}\Longleftrightarrow b_{m-1}+2\in A'\\
&\Longleftrightarrow& \cdots \Longleftrightarrow b_{m-1}+t-1\in A',
\end{eqnarray*}
\begin{eqnarray*}
k-1+\frac{b_{m-1}}{2}\in A' &\overset{(\ref{e2.30})}\Longleftrightarrow& k-2+\frac{b_{m-1}}{2}\not\in A' \overset{(\ref{e2.31})}\Longleftrightarrow k+\frac{b_{m-1}}{2}\in A'\\
&\Longleftrightarrow& \cdots \Longleftrightarrow 2k-3-t\in A'.
\end{eqnarray*}

If $t=2$, then $k-1+\frac{b_{m-1}}{2}\in A'$, thus
\begin{eqnarray}\label{e2.32}[b_{m-1}+1, 2k-4]\cap A'=\{b_{m-1}+1\}\cup \left[k-1+\frac{b_{m-1}}{2}, 2k-5\right]. \end{eqnarray}
By Lemma \ref{L2-2} and Lemma \ref{L2-3}, we have $\frac{b_{m}}{2}\in A'$ and $b_{m-1}+1\leqslant \frac{b_{m}}{2}\leqslant k-2$, thus
$$\frac{b_{m}}{2}=b_{m-1}+1. $$
Since $a_{i}<2i$ for all $i=1, \ldots, k-2$, we have $1\in A'$.
Noting that $2k-5\in A'$, we have $2k-4\in 2^{\wedge}A'$, thus $b_{m}<2k-4$. By (\ref{e2.32}) we have
$$b_{m-1}+2\leqslant b_{m}\leqslant k-2+\frac{b_{m-1}}{2}. $$
By Lemma \ref{L2-2} (ii), we have
$$|[0, b_{m}]\cap A'|=\frac{b_{m}}{2}+1=b_{m-1}+2, $$
thus
$$|[0, b_{m-1}+1]\cap A'|=|[0, b_{m}]\cap A'|=b_{m-1}+2. $$
Hence, $[0, b_{m-1}+1]\subseteq A'$, which contradicts with $b_{m-1}\not\in A'$.

If $t=k-1-\frac{b_{m-1}}{2}$, then
$$[b_{m-1}+1, 2k-4]\cap A'=\left[b_{m-1}+1, k-2+\frac{b_{m-1}}{2}\right], $$
thus
$$\left[2b_{m-1}+3, 2k-5+b_{m-1}\right]\subseteq 2^{\wedge}A', $$
which implies that $b_{m}=2b_{m-1}+2$. If $b_{m}<2k-4$, then $b_{m}+1\in A'$, a contradiction. Thus $b_{m}=2k-4$, and then $b_{m-1}=k-3$.
By Lemma \ref{L2-2} (ii), we have
$$A'=\left[0, \frac{k-3}{2}\right]\cup \left[k-2, \frac{3(k-3)}{2}+1\right], $$
thus $B=\{k-3, 2k-4\}$.

Now, we assume that $3\leqslant t\leqslant k-2-\frac{b_{m-1}}{2}$. Then
\begin{equation}\label{e2.33}[b_{m-1}+1, 2k-4]\cap A'=[b_{m-1}+1, b_{m-1}+t-1]\cup \left[k-1+\frac{b_{m-1}}{2}, 2k-3-t\right], \end{equation}
thus
$$b_{m-1}+t\leqslant b_{m}\leqslant k-2+\frac{b_{m-1}}{2} \text{ or } 2k-2-t\leqslant b_{m}\leqslant 2k-4. $$

{\bf Case 1. }$2k-2-t\leqslant b_{m}\leqslant 2k-4$. Then $b_{m}=2k-4$. If not, $b_{m}+1\in A'$, a contradiction.
By (\ref{e2.33}) and $\frac{b_{m}}{2}=k-2\in A'$, we have
$$b_{m-1}+1\leqslant k-2\leqslant b_{m-1}+t-1, $$
which implies that $k-2=b_{m-1}+1$ or $b_{m-1}+t-1$.
If $k-2=b_{m-1}+1$, then by Lemma \ref{L2-2} (i) and (\ref{e2.33}), we have
$$\frac{b_{m-1}}{2}+\left(k-1+\frac{b_{m-1}}{2}\right)=k-1+b_{m-1}=2k-4\in 2^{\wedge}A', $$
a contradiction.

Let $k-2=b_{m-1}+t-1$. Again by (\ref{e2.33}) and Lemma \ref{L2-2}, we have
$$A'=\left[0, \frac{k-3}{3}\right]\cup\left[\frac{2k-3}{3}, k-2\right]\cup\left[\frac{4k-6}{3}, \frac{5k-12}{3}\right], $$
thus $B=\left\{\frac{2k-6}{3}, 2k-4\right\}$.

{\bf Case 2. }$b_{m-1}+t\leqslant b_{m}\leqslant k-2+\frac{b_{m-1}}{2}$. Since $b_{m}\not\in A'$ and $b_{m}+1\in A'$, we have
$b_{m}=k-2+\frac{b_{m-1}}{2}$. By (\ref{e2.33}) we have $\frac{b_{m}}{2}=b_{m-1}+1$ or $\frac{b_{m}}{2}=b_{m-1}+t-1$.
Similar to Case 1, we have
$$A'=\left[0, \frac{k-4}{3}\right]\cup\left[\frac{2k-5}{3}, k-3\right]\cup\left[\frac{4k-7}{3}, \frac{5k-11}{3}\right], $$
thus $B=\left\{\frac{2k-8}{3}, \frac{4k-10}{3}\right\}$.

This completes the proof of Proposition \ref{P4}.
\end{proof}



\end{document}